\documentclass[11pt]{article}
\usepackage{amsfonts, amsmath, amssymb, amscd, amsthm, color, graphicx, mathrsfs, wasysym, setspace, mdwlist, calc,float}
\usepackage{hyperref}
\newtheorem{thm}{Theorem}[section]
\newtheorem{cor}[thm]{Corollary}
\newtheorem{lem}[thm]{Lemma}
\newtheorem{prop}[thm]{Proposition}
\newtheorem{prob}[thm]{Question}

\theoremstyle{definition}
\newtheorem{defn}[thm]{Definition}
\theoremstyle{remark}
\newtheorem{rem}[thm]{Remark}
\newtheorem{ex}[thm]{Example}

\renewcommand{\d }{{\rm d} }

\newcommand{\dzh }{\d_{Z \sqcup\mathcal H}\, }
\newcommand{\Hl }{\{ H_\lambda \} _{\lambda \in \Lambda } }
\newcommand{\e }{\varepsilon }
\renewcommand{\kappa }{\varkappa}

\begin{document}

\title{Acylindrical group actions on quasi-trees }
\author{S. Balasubramanya}
\date{}
\maketitle

\abstract A group G is acylindrically hyperbolic if it admits a non-elementary acylindrical action on a hyperbolic space. We prove that every acylindrically hyperbolic group $G$ has a generating set $X$ such that the corresponding Cayley graph $\Gamma$ is a (non-elementary) quasi-tree and the action of $G$ on $\Gamma$ is acylindrical. Our proof utilizes the notions of hyperbolically embedded subgroups and projection complexes. As an application, we obtain some new results about hyperbolically embedded subgroups and quasi-convex subgroups of acylindrically hyperbolic groups. 

\section{Introduction}
Recall that an isometric action of a group $G$ on a metric space $(S,\d)$ is {\it acylindrical} if for every $\e>0$ there exist $R,N>0$
such that for every two points $x,y$ with $\d(x,y)\ge R$, there are at most $N$ elements $g\in G$ satisfying
$$
\d(x,gx)\le \e \;\;\; {\rm and}\;\;\; \d(y,gy) \le \e.
$$
Obvious examples are provided by geometric (i.e., proper and cobounded) actions; note, however, that acylindricity is a much weaker condition.

A group $G$ is called \emph{acylindrically hyperbolic} if it admits a non-elementary acylindrical action on a hyperbolic space. Over the last few years, the class of acylindrically hyperbolic groups has received considerable attention. It is broad enough to include many examples of interest, e.g., non-elementary hyperbolic and relatively hyperbolic groups, all but finitely many mapping class groups of punctured closed surfaces, $Out(F_n)$ for $n\ge 2$, most $3$-manifold groups, and finitely presented groups of deficiency at least $2$. On the other hand, the existence of a non-elementary acylindrical action on a hyperbolic space is a rather strong assumption, which allows one to prove non-trivial results.  In particular, acylindrically hyperbolic groups share many interesting properties with non-elementary hyperbolic and relatively hyperbolic groups. For details we refer to \cite{DGO,MO,ah,Osi15} and references therein.

The main goal of this paper is to answer the following.

\begin{prob}\label{q}
Which groups admit non-elementary cobounded acylindrical actions on quasi-trees?
\end{prob}

In this paper, by a quasi-tree we mean a connected graph which is quasi-isometric to a tree. Quasi-trees form a very particular subclass of the class of all hyperbolic spaces. From the asymptotic point of view, quasi-trees are exactly ``1-dimensional hyperbolic spaces".  

The motivation behind our question comes from the following observation. If instead of cobounded acylindrical actions we consider cobounded proper (i.e., geometric) ones, then there is a crucial difference between the groups acting on hyperbolic spaces and quasi-trees. Indeed a group $G$ acts geometrically on a hyperbolic space if and only if $G$ is a hyperbolic group. On the other hand, Stallings theorem on groups with infinitely many ends and Dunwoodys accessibility theorem implies that groups admitting geometric actions on quasi-trees are exactly virtually free groups. Yet another related observation is that acylindrical actions on unbounded locally finite graphs are necessarily proper. Thus if we restrict to quasi-trees of bounded valence in Question \ref{q}, we again obtain the class of virtually free groups. Other known examples of groups having non-elementary, acylindrical and cobounded actions on quasi-trees include groups associated  with special cube complexes and right angled artin groups (see \cite{BHS}, \cite{Hagen}, \cite{KK}). 

Thus one could expect that the answer to Question \ref{q} would produce a proper subclass of the class of all acylindrically hyperbolic groups, which generalizes virtually free groups in the same sense as acylindrically hyperbolic groups generalize hyperbolic groups. Our main result shows that this does not happen.

\begin{thm}\label{main}
Every acylindrically hyperbolic group admits a non-elementary cobounded acylindrical action on a quasi-tree.
\end{thm}

In other words,  being acylindrically hyperbolic is equivalent to admitting a non-elementary acylindrical action on a quasi-tree are equivalent. Although this result does not produce any new class of groups, it can be useful in the study of acylindrically hyperbolic groups and their subgroups. In this paper we concentrate on proving Theorem \ref{main} and leave applications for  future papers to explore (for some applications, see \cite{MO}).  

It was known before that every acylindrically hyperbolic group admits a non-elementary cobounded action on a quasi-tree satisfying the so-called \emph{weak proper discontinuity} property, which is weaker than acylindricity. Such a quasi-tree can be produced by using projection complexes introduced by Bestvina-Bromberg-Fujiwara in \cite{BBF}.  To the best of our knowledge, whether the corresponding action is acylindrical is an open question. The main idea of the proof of Theorem \ref{main} is to combine the Bestvina-Bromberg-Fujiwara approach with an `acylindrification' construction from \cite{ah} in order to make the action acylindrical. An essential role in this process is played by the notion of a hyperbolically embedded subgroup introduced in \cite{DGO} - this fact is of independent interest since it provides a new setting for the application of the Bestvina-Bromberg-Fujiwara construction. 

The above mentioned construction has been applied in the setting of geometrically separated subgroups (see \cite{DGO}), but hyperbolically embedded subgroups do not necessarily satisfy this condition. Nevertheless, it is possible to employ them in this construction, possibly with interesting applications. If fact, we prove much stronger results in terms of hyperbolically embedded subgroups (see Theorem \ref{3 condns}) of which Theorem \ref{main} is an easy consequence, and derive an application in this paper which is stated below (see Corollary \ref{appln}). \begin{cor}\label{introv} Let $G$ be a group. If $H \leq K \leq G$ , $H$ is countable and $H$ is hyperbolically embedded in $G$, then $H$ is hyperbolically embedded in $K$.\end{cor} We would like to note that the above result continues to hold even when we have a finite collection $\{H_1, H_2,...,H_n\}$ of hyperbolically embedded subgroups in $G$ such that $H_i \leq K$ for all $i=1,2,...,n$. Interestingly, A.Sisto obtains a similar result in \cite{sim}, Corollary 6.10. His result does not require $H$ to be countable, but under the assumption that $H \cap K$ is a virtual retract of $K$, it  states that $H \cap K \hookrightarrow_h K$. Although similar, these two theorems are distinct in the sense that neither follows from the other.  

Another application of Theorem \ref{3 condns} is to the case of finitely generated subgroups, as stated below (see Corollary \ref{newqc}). \begin{cor} Let $H$ be a finitely generated subgroup of an acylindrically hyperbolic group $G$. Then there exists a subset $X \subset G$ such that \begin{itemize} 
\item[(a)] $\Gamma(G, X)$ is hyperbolic, non-elementary and acylindrical 
\item[(b)] $H$ is quasi-convex in $\Gamma(G,X)$
\end{itemize} 
\end{cor} 
The above result indicates that in order to develop a theory of quasi-convex subgroups in acylindrically hyperbolic groups, the notion of quasi-convexity is not sufficient, i.e., a stronger set of conditions is necessary in order to prove results similar to those known for quasi-convex subgroups in hyperbolic groups. For example, using Rips' construction from \cite{Rips} and the above corollary, one can easily construct an example of an infinite, infinite index, normal subgroup in an acylindrically hyperbolic group, which is quasi-convex with respect to some non-elementary acylindrical action. \\

\textbf{Acknowledgements:}  My heartfelt gratitude to my advisor Denis Osin for his guidance and support, and to Jason Behrstock and Yago Antolin Pichel  for their remarks. My sincere thanks to Bryan Jacobson for his thorough proof-reading and comments on this paper. 

\section{Preliminaries}
We recall some definitions and theorems which we will need to refer to.

\subsection{Relative Metrics on subgroups}
\begin{defn}[Relative metric] \label{rm} Let $G$ be a group  and $ \Hl $ a fixed collection of subgroups of $G$. Let $X \subset G$ such that $G$ is generated by $X$ along with the union of all $\Hl$. Let $\mathcal{H} =\bigsqcup_{\lambda \in \Lambda}  H_\lambda$.  We denote the corresponding Cayley graph of $G$ (whose edges are labeled by elements of $X \sqcup \mathcal{H}$) by $\Gamma(G, X \sqcup \mathcal{H})$. 

\begin{rem} \label{bigons} It is important that the union in the definition above is disjoint. This disjoint union leads to the following observation $\colon$ for every $h \in H_i \cap H_j$, the alphabet $\mathcal{H}$ will have two letters representing $h$ in $G$, one from $H_i$ and another from $H_j$. It may also be the case that a letter from $\mathcal{H}$ and a letter from $X$ represent the same element of the group $G$. In this situation, the corresponding Cayley graph $\Gamma(G, X \sqcup \mathcal{H})$ has bigons (or multiple edges in general) between the identity and the element, one corresponding to each of these letters.\end{rem}

We think of $\Gamma(H_\lambda,H_\lambda)$  as a complete subgraph in $\Gamma(G, X \sqcup \mathcal{H})$. A path $p$ in $\Gamma(G, X \sqcup \mathcal{H})$ is said to be \textit{$\lambda$-admissible} if it contains no edges of the subgraph  $\Gamma(H_\lambda,H_\lambda)$. In other words, the path $p$ does not travel through $H_\lambda$ in the Cayley graph. Using this notion, we can define a metric, known as the relative metric $\widehat{d}_\lambda : H_\lambda \times H_\lambda \rightarrow [0, +\infty ] $ by setting  $\widehat{d}_\lambda(h,k)$ for $h, k \in H_\lambda$ to be the length of the shortest admissible path in $\Gamma(G, X \sqcup \mathcal{H})$ that connects $h$ to $k$. If no such path exists, we define $\widehat{d}_\lambda(h,k) = +\infty$. It is easy to check that $\widehat{d}_\lambda$ satisfies all the conditions to be a metric function. 
\end{defn}

\begin{defn} Let $q$ be a path in the Cayley graph of $\Gamma(G, X \sqcup \mathcal{H})$ and $p$ be a nontrivial subpath of $q$. $p$ is said to be an \textit{$H_\lambda$- subpath} if the label of $p$ (denoted \textbf{Lab}($p$)) is a word in the alphabet $H_\lambda$. Such a subpath is further called an \textit{$H_\lambda$-component} if it is not contained in a longer $H_\lambda$-subpath of $q$. If $q$ is a loop, we must also have that $p$ is not contained in a longer $H_\lambda$-subpath of any cyclic shift of $q$. We refer to an $H_\lambda$-component of $q$ (for some $\lambda \in \Lambda$) simply by calling it a \textit{component} of $q$. We note that on a geodesic, $H_\lambda$- components must be single $H_\lambda$-edges.

Let $p_1, p_2$ be two $H_\lambda$ components of  a path $q$ for some $\lambda \in \Lambda$. $p_1$ and $p_2$ are said to be \textit{connected} if there exists a path $p$ in $\Gamma(G, X \sqcup \mathcal{H})$  such that \textbf{Lab}($p$) is a word consisting only of letters from $H_\lambda$, and $p$ connects some vertex of $p_1$ to some  vertex of $p_2$. In algebraic terms, this means that all vertices of $p_1$ and $p_2$ belong to the same (left) coset of $H_\lambda$. We refer to a component of a path $q$ as \textit{isolated} if it is not connected to any other component of $q$.
\end{defn}

If $p$ is a path, we denote its initial point by $p_{-}$ and its terminating point by $p_{+}$. 

\begin{lem}[\cite{DGO}, Proposition 4.13] \label{isolated}  Let $G$ be a group  and $ \Hl $ a fixed collection of subgroups in $G$. Let  $X \subset G$ such that $G$ is generated by $X$ together with the union of all $\Hl$. Then there exists a constant $C > 0$ such that for any n-gon p with geodesic sides in  $\Gamma(G, X \sqcup \mathcal{H})$, any $\lambda \in \Lambda$, and any isolated $H_\lambda$ component a of p,  $\widehat{d}_\lambda(a_{-}, a_{+}) \leq Cn$.
\end{lem}

\subsection{Hyperbolically embedded subgroups}

Hyperbolically embedded subgroups will be our main tool in constructing the quasi-tree. The notion has been taken from \cite{DGO}. We recall the definition here. 

\begin{defn}[Hyperbolically embedded subgroups] Let $G$ be a group. Let $X$ be a (not necessarily finite) subset of $G$ and  let $\{H_\lambda \}_{\lambda \in \Lambda}$ be a collection of subgroups of $G$. We say that $\{H_\lambda \}_{\lambda \in \Lambda}$ is \textit{hyperbolically embedded} in $G$ with respect to $X$  (denoted by $\{H_\lambda\}_{\lambda \in \Lambda} \hookrightarrow_{h} (G, X)$ ) if the following conditions hold $\colon$
\begin{enumerate}
\item[(a)] The group $G$ is generated by $X$ together with the union of all $\{H_\lambda\}_{\lambda \in \Lambda}$ 
\item[(b)] The Cayley graph $\Gamma(G, X \sqcup \mathcal{H})$ is hyperbolic, where $\mathcal{H} = \bigsqcup_{\lambda \in \Lambda} H_\lambda$. 
\item[(c)] For every $\lambda \in \Lambda$, the metric space $(H_\lambda, \widehat{d}_\lambda)$ is proper, i.e., every ball of finite radius has finite cardinality.
\end{enumerate}
\end{defn}

Further we say that $\{H_\lambda\}_{\lambda \in \Lambda}$  is hyperbolically embedded in $G$ (denoted by $\{H_\lambda\}_{\lambda \in \Lambda} \hookrightarrow_{h} G$) if
$\{H_\lambda\}_{\lambda \in \Lambda} \hookrightarrow_{h} (G, X) $ for some $X \subseteq G$. The set $X$ is called a \textit{relative generating set}. 

\begin{figure}
\centering
\begin{minipage}{.49\textwidth}
  \centering
  \def\svgscale{0.325}
  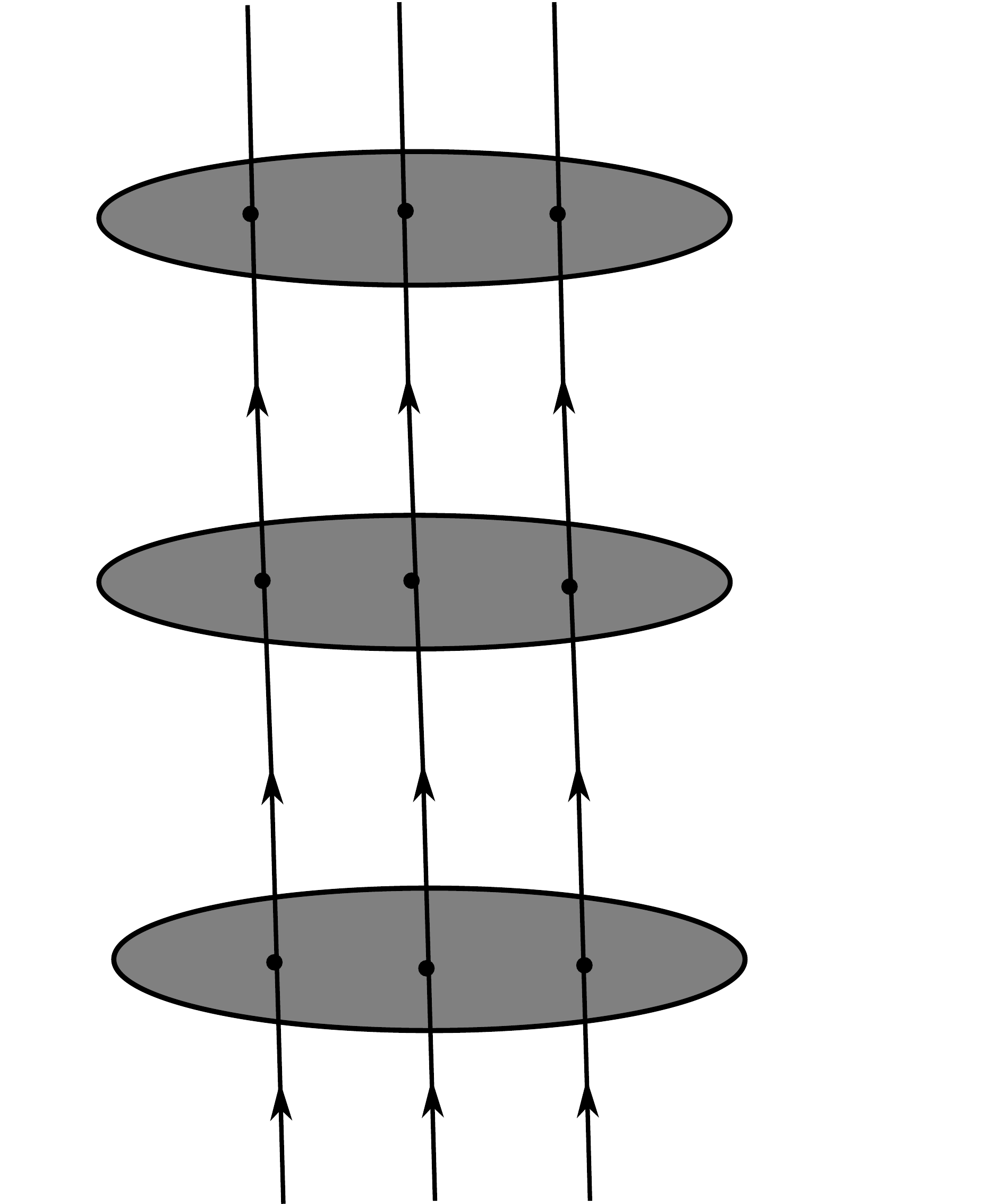
\caption{$H \times \mathbb{Z}$}
\label{1a}
\end{minipage}%
\begin{minipage}{.49\textwidth}
  \centering
  \def\svgscale{0.325}
  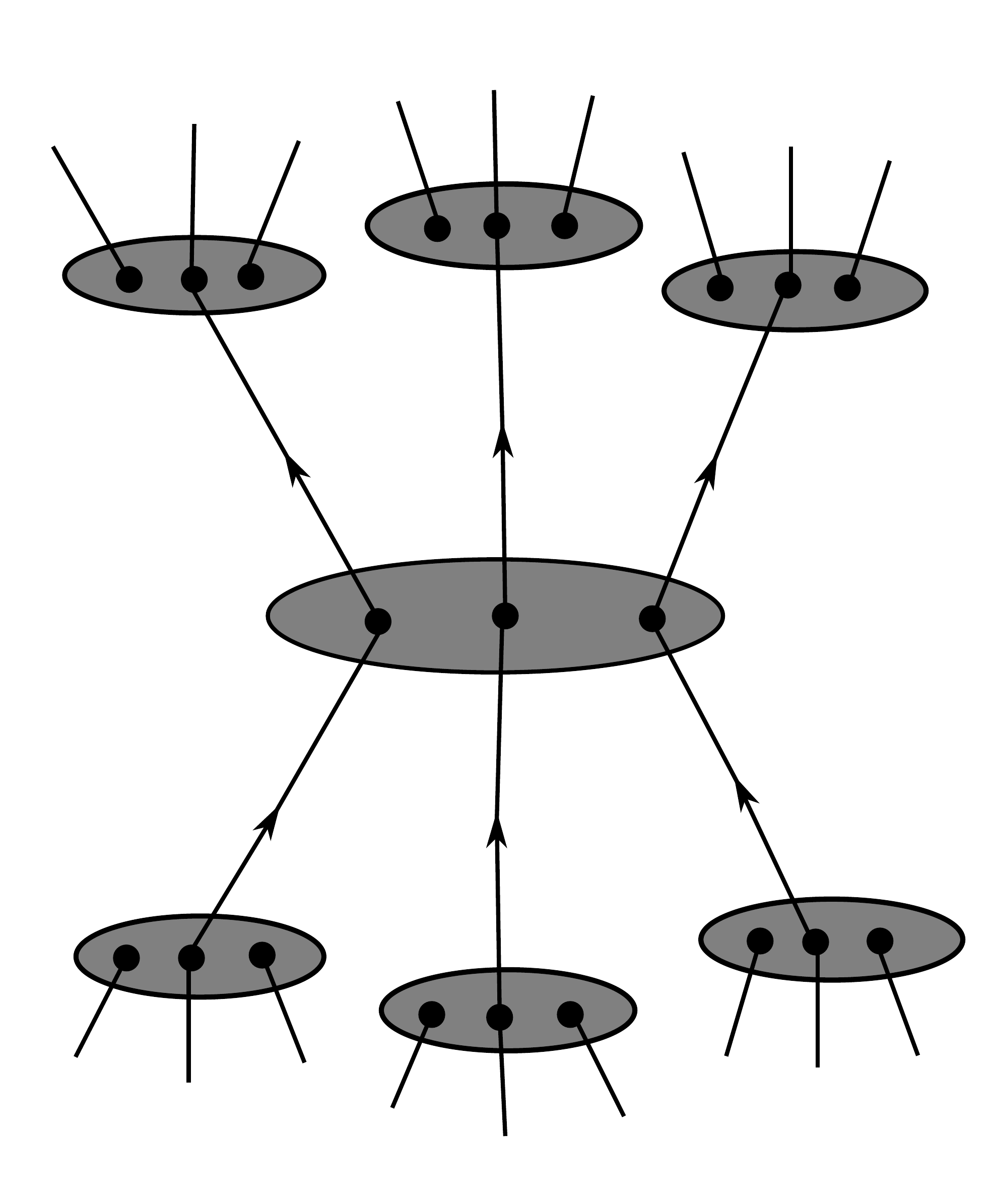
 \caption{$H * \mathbb{Z}$}
\label{1b}
\end{minipage}
\end{figure}

Since the notion of a hyperbolically embedded subgroup plays a crucial role in this paper, we include two examples borrowed from \cite{DGO}.

\begin{ex} Let $G = H \times \mathbb{Z}$ and $\mathbb{Z} = \langle x\rangle$. Let $X = \{x\}$. Then $\Gamma(G, X \sqcup H)$ is quasi-isometric to a line and is hence hyperbolic. The corresponding relative metric satisfies the following inequality which is easy to see from the Cayley graph (see Fig.\ref{1a}) $\colon \widehat{d}(h_1, h_2) \leq 3$ for every $h_1, h_2 \in H$. Indeed if $\Gamma_H$ denotes the Cayley graph $\Gamma(H, H)$, then in its shifted copy $x\Gamma_H$, there is an edge $e$ connecting $xh_1$ to $xh_2$ (labeled by $h_{1}^{-1}h_2 \in H$) . There is thus an admissible path of length $3$ connecting $h_1$ to $h_2$. We conclude that if $H$ is infinite, then $H$ is not hyperbolically embedded in $(G,X)$, since the relative metric will not be proper. 
In this example, one can also note that the admissible path from $h_1$ to $h_2$ contains an $H$-subpath, namely the edge $e$, which is also an $H$-component of this path. 
\end{ex}

\begin{ex} Let $G = H * \mathbb{Z}$ and $\mathbb{Z} = \langle x \rangle$.  As in the previous example, let $X = \{x\}$. In this case $\Gamma(G, X \sqcup H)$ is quasi-isometric to a tree (see Fig.\ref{1b}) and it is easy to see that $\widehat{d}(h_1, h_2) = \infty$ unless $h_1 = h_2$. This means that every ball of finite radius in the relative metric has cardinality $1$. We can thus conclude that $H \hookrightarrow_h (G,X)$.
\end{ex}

\subsection{A slight modification to the relative metric} 
The aim of this section is to modify the relative metric on countable subgroups that are hyperbolically embedded, so that the resulting metric takes values only in $\mathbb{R}$, i.e., is finite valued. This will be of importance in section \ref{constr}. The main result of this section is the following.
\begin{thm}\label{tilde} Let $G$ be a group. Let $H < G$ be countable, such that $H \hookrightarrow_h G$. Then there exists a metric $\widetilde{d} \colon H \times H \rightarrow \mathbb{R}$, such that \begin{itemize} \item[(a)] $\widetilde{d} \leq \widehat{d}$ \item[(b)] $\widetilde{d}$ is proper, i.e., every ball of finite radius has finitely many elements. \end{itemize} \end{thm} 

\begin{proof} There exists a collection of finite, symmetric (closed under inverses) subsets $\{F_i\}$ of H such that $H = \bigcup_{i=1}^{\infty} F_i$ and ${1} \subseteq F_1 \subseteq F_2 \subseteq ...$

Let $\widehat{d}$ be the relative metric on $H$. Let $H_0 = \{ h\in H \mid \widehat{d}(1, h) < \infty \}$. 

Define a function $w : H \rightarrow \mathbb{N}$ as \begin{center} $w(h) = \left \{ 
\begin{tabular}{cc}
  $\widehat{d}(1, h)$  & , if $h \in H_0$ \\
  min $\{ i \mid h \in F_i\}$ & , otherwise
  \end{tabular}
\right. $
\end{center}

Since $F_i$'s are symmetric, $w(h) = w(h^{-1})$ for all $h \in H$. Define a function $l$ on H as follows- for every word $u = x_1x_2...x_k$ in the elements of $H$, set $$l(u) = \sum_{i=1}^{k} w(x_i).$$ Set a length function on $H$ as 
\begin{center}$|g|_w = min \{l(u) \mid $ u is a word in the elements of $H$ that represents $g\}$, \end{center} for each $g$ in $H$. We can now define a metric $d_w \colon H \times H \rightarrow \mathbb{N}$ as $$d_w(g,h) = |g^{-1}h|_w.$$ It is easy to check that $d_w$ is a (finite valued) well-defined metric, which satisfies $d_w(1,h) \leq w(h)$ for all $h \in H$. 

It remains to show that $d_w$ is proper. Let $N \in \mathbb{N}$. Suppose $h \in H$ such that $w(h) \leq N$. If $h \in H_0$, then $\widehat{d}(1, h) \leq N$ which implies that there are finitely many choices for $h$, since $\widehat{d}$ is proper. If $h \notin H_0$, then $h \in F_i$ for some minimal $i$. But each $F_i$ is a finite set, so there are finitely many choices for $h$. Thus $\large| \{ h \in H \mid w(h) \leq N \}\large| < \infty$ for all $N \in \mathbb{N}$. This implies $d_w$ is proper. 

Indeed, if $y \neq 1$ is such that $|y|_w \leq n$, then there exists a word $u$, written without the identity element (which has weight zero), representing $y$ in the alphabet $H$ such that $u = x_1 x_2 ...x_r $ and $\sum_{i =1} ^{r} w(x_i) \leq n$. Since $w(x_i) \geq 1$ for every $x_i \neq 1$, $r \leq n$. Further, $w(x_i) \leq n$ for all $i$. Thus $x_i \in \{ x \in H \mid w(x) \leq n \}$ for all $i$. So there only finitely many choices for each $x_i$, which implies there are finitely many choices for $y$. By definition, $d_w \leq \widehat{d}$. So we can set $\widetilde{d} = d_{w}$.
\end{proof}

\subsection{Acylindrically Hyperbolic Groups}
In the following theorem $\partial$ represents the Gromov boundary. 

\begin{thm} \label{ah} For any group G, the following are equivalent.
\begin{enumerate}
\item[(AH$_1$)] There exists a generating set X of G such that the corresponding Cayley graph $\Gamma$(G,X) is hyperbolic, $|\partial \Gamma(G,X)| \geq 2 $, and the natural action of G on $\Gamma$(G,X) is acylindrical.

\item[(AH$_2$)] G admits a non-elementary acylindrical action on a hyperbolic space.

\item[(AH$_3$)] G contains a proper infinite hyperbolically embedded subgroup.
\end{enumerate}
\end{thm}

It follows from the definitions that $(AH_1) \Rightarrow (AH_2)$. The implication
$(AH_2) \Rightarrow (AH_3)$  is non-trivial and was proved in \cite{DGO}. The implication $(AH_3) \Rightarrow (AH_1)$ was proved in \cite{ah}.

\begin{defn} We call a group G \textit{acylindrically hyperbolic} if it satisfies any of the equivalent conditions (AH$_1$)-(AH$_3$) from Theorem \ref{ah}. \end{defn}

\begin{lem}[\cite{DGO}, Corollary 4.27] \label{symmdiff} Let G be a group, $\{ H_\lambda\}_{\lambda \in \Lambda}$ a collection of subgroups of G, and $X_1$ and $X_2$ be relative generating sets. Suppose that $| X_1 \Delta X_2| < \infty$. Then  $\{ H_\lambda\}_{\lambda \in \Lambda} \hookrightarrow_h (G, X_1)$ if and only if  $\{ H_\lambda\}_{\lambda \in \Lambda} \hookrightarrow_h (G, X_2)$.
\end{lem}

\begin{thm}[\cite{ah}, Theorem 5.4] \label{DO}  Let G be a group, $ \{ H_\lambda \}_{\lambda \in \Lambda}$ a finite collection of subgroups of G, X a subset of G. Suppose that $ \{ H_\lambda \}_{\lambda \in \Lambda} \hookrightarrow_{h} (G,X)$. Then there exists $Y \subset G$ such that the following conditions hold.
\begin{enumerate}
\item[(a)] $X \subset Y$
\item[(b)] $ \{ H_\lambda \}_{\lambda \in \Lambda} \hookrightarrow_{h} (G,Y)$. In particular, the Cayley graph $\Gamma(G, Y \sqcup \mathcal{H})$ is hyperbolic.
\item[(c)] The action of G on $\Gamma (G, Y \sqcup \mathcal{H})$ is acylindrical.
\end{enumerate}
\end{thm}

\begin{defn} Let $(X,d_X)$ and $(Y, d_Y)$ be two metric spaces. A map $\phi\colon X \rightarrow Y$ is said to be a \textit{($\lambda$,C)-quasi-isometry} if there exist  constants $\lambda > 1, C >0$  such that \begin{enumerate}
\item[(a)] $\frac{1}{\lambda}d_X(a,b) - C \leq d_Y(\phi(a), \phi(b)) \leq \lambda d_X(a,b) + C$, for all $ a,b \in X$ and 
\item[(b)] $Y$ is contained in the $C$-neighborhood of $\phi(X)$. 
\end{enumerate}
The spaces $X$ and $Y$ are said to be \textit{quasi-isometric} if such a map $\phi\colon X \rightarrow Y$ exists. It is easy to check that being quasi-isometric is an equivalence relation. If the map $\phi$ satisfies only condition (a), then it is said to be a \textit{($\lambda$,C)-quasi-isometric embedding}.\end{defn}

\begin{defn} A graph $\Gamma$ with the combinatorial metric $d_\Gamma$ is said to be a \textit{quasi-tree}  if it is quasi-isometric to a tree $T$. 
\end{defn}

\begin{defn} A \textit{quasi-geodesic} is a quasi-isometric embedding of an interval (bounded or unbounded) $I \subseteq \mathbb{R}$ into a metric space $X$. Note that geodesics are $(1,0)$-quasi-geodesics. By slight abuse of notation, we may identify the map that defines a quasi-geodesic with its image in the space. 
\end{defn}

\begin{thm}[\cite{JFM}, Theorem 4.6, Bottleneck property] \label{bottleneck} Let Y be a geodesic metric space. The following are equivalent.
\begin{enumerate}
\item[(a)] Y is quasi-isometric to some simplicial  tree $\Gamma$
\item[(b)] There is some $\mu > 0$ so that for all $x, y \in$ Y, there is a midpoint $m=m(x,y)$ with d(x,m) = d(y,m) = $\frac{1}{2}$d(x,y) and the property that any path from x to y must pass within less than $\mu$ of the point m.
\end{enumerate}
\end{thm}

We remark that if $m$ is replaced with any point $p$ on a geodesic between $x$ and $y$, then the property that any path from $x$ to $y$ passes within less than $\mu$ of the point $p$ still follows from (a), as proved below in Lemma \ref{bpr}.  We will need the following lemma. 

\begin{lem}\label{bottlenecklemma}[\cite{bottlenecklemma}, Proposition 3.1] For all $\lambda \geq 1, C \geq 0, \delta \geq 0$, there exists an $R= R(\delta, \lambda, C)$ such that if $X$ is a $\delta$-hyperbolic space, $\gamma$ is a $(\lambda, C)$-quasi-geodesic in $X$, and $\gamma '$ is a geodesic segment with the same end points, then  $\gamma '$ and $\gamma$ are Hausdorff distance less than $R$ from each other. \end{lem}

\begin{lem}\label{bpr} If Y is a quasi-tree, then there exists $\mu>0 $ such that for any point $z$ on a geodesic connecting two points, any other path between the same end points passes within $\mu$ of $z$. \end{lem}
\begin{proof}
Let $T$ be a tree and $q\colon Y \rightarrow T$ be the $(\lambda,C)$ quasi-isometry. Let $d_Y$ and $d_T$ denote the metrics in the spaces $Y$ and $T$ respectively. Note that since $T$ is 0-hyperbolic, $Y$ is $\delta$-hyperbolic for some $\delta$.

Let $x, y$ be two points in $Y$, joined by a geodesic $\gamma$. Let $z$ be any point of $\gamma$, and let $\alpha$ be another path from $x$ to $y$. Let $V$ denote the vertex set of $\alpha$, ordered according to the geodesic $\gamma$. Take its image $q(V)$ and connect consecutive points by geodesics (of length at most $\lambda + C$) to get a path $\beta$ in $T$ from $q(x)$ to $q(y)$. Then the unique geodesic $\sigma$ in $T$ must be a subset of $\beta$. Since $q(V) \subset q\circ\alpha$, we get that any point of $\sigma$ s at most $\lambda + C$ from $q\circ\alpha$. Also, $q\circ\gamma$ is a ($\lambda$, C)-quasi-isometric embedding of an interval, and hence a $(\lambda, C)$-quasi-geodesic. Thus, by Lemma \ref{bottlenecklemma} the distance from $q(z)$ to $\sigma$ is less than $R= R(0, \lambda, C).$

Let $p$ be the point on $\sigma$ closest to $q(z)$. There is a point $w \in Y$ on $\alpha$ such that $d(q(w), p) \leq \lambda +C$. Since $d(p, q(z)) < R$, we have $d(q(w), q(z)) \leq \lambda + C+ R$. Thus $$d(z, w) \leq \lambda^2 +2\lambda C +R\lambda.$$ Thus $\alpha$ must pass within $\mu = \lambda^2 +2\lambda C +R\lambda$ of the point $z$.
\end{proof}

\subsection{ A modified version of Bowditch's lemma} \label{modB}

In this section, $\mathcal{N}_k(x)$ denotes the closed $k$-neighborhood of a point $x$ in a metric space. The following theorem  will be used in Section 5.  Part (a) is a simplified form of a result taken from \cite{kapo}, which is infact derived from a hyperbolicity criterion developed by Bowditch in \cite{bow}. 

\begin{thm} \label{B} Let $\Sigma$ be a hyperbolic graph, and $\Delta$ be a graph obtained from $\Sigma$ by adding edges.  \begin{enumerate} \item[(a)]\cite{bow} Suppose there exists $M >0$ such that for all vertices $x,y \in \Sigma$ joined by an edge in $\Delta$ and for all geodesics p in $\Sigma$ between $x$ and $y$, all vertices of p lie in an M-neighborhood of $x$, i.e., $p \subseteq \mathcal{N}_{M}(x)$ in $\Delta$.  Then $\Delta$ is also hyperbolic, and there exists a constant $k$ such that for all vertices x,y $\in \Sigma$, every geodesic q between x and y in $\Sigma$  lies in a k-neighborhood in $\Delta$ of every geodesic in $\Delta$ between x and y.
\item[(b)]  If, under the assumptions of  (a), we additionally assume that $\Sigma$ is a quasi-tree, then $\Delta$ is also a quasi-tree.
\end{enumerate} \end{thm}

\begin{figure}
\centering
\def\svgscale{0.4}
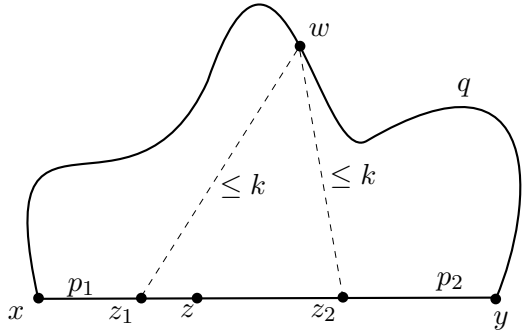
\caption{Corresponding to Lemma \ref{neighborhood}}
\label{flip}
\end{figure}

\begin{lem} \label{neighborhood} Let p,q be two paths in a metric space S between points x and y, such that p is a geodesic and  q $\subseteq \mathcal{N}_k(p)$. Then $p$ $\subseteq \mathcal{N}_{2k}(q)$.
\end{lem}
\begin{proof} Let $z$ be any point on $p$. Let $p_1, p_2$ denote the segments of the geodesic $p$ with end points $x, z$ and $z, y$ respectively. 

Define a function $f \colon q \rightarrow \mathbb{R}$ as $f(s) = d(s, p_1) - d(s, p_2)$. Then $f$ is a continuous function. Further, $f(x) < 0$ and $f(y) > 0$. By the intermediate value theorem, there exists a point $w$ on $q$ such that $f(w) =0$. Thus $d(w, p_1) = d(w, p_2)$ (see Fig.\ref{flip}). 
Let $z_1$ (resp. $z_2$) be a point of $p_1$ (resp. $p_2$) such that $d(p_i, w) = d(z_i,w)$ for $i=1,2$. Then $d(z_1,w)= d(z_2,w)$. By the hypothesis, $d(w,p) = min\{d(w,p_1), d(w,p_2)\} \leq k$. So we get that $d(w,p_1) = d(w,p_2) \leq k$. Thus $d(z_1, z_2) \leq 2k$, which implies $d(z,w) \leq 2k$. 
\end{proof}

\begin{proof}[Proof of Theorem \ref{B}] We proceed with the proof of part (b). 

We prove that $\Delta$ is a quasi-tree by verifying the bottleneck property from Theorem \ref{bottleneck}. Let $d_\Sigma$ (resp. $d_\Delta$) denote the distance in the graph $\Sigma$ (resp. $\Delta$). Note that the vertex sets of the two graphs are equal. 

\begin{figure}
\centering
\def\svgscale{0.5}
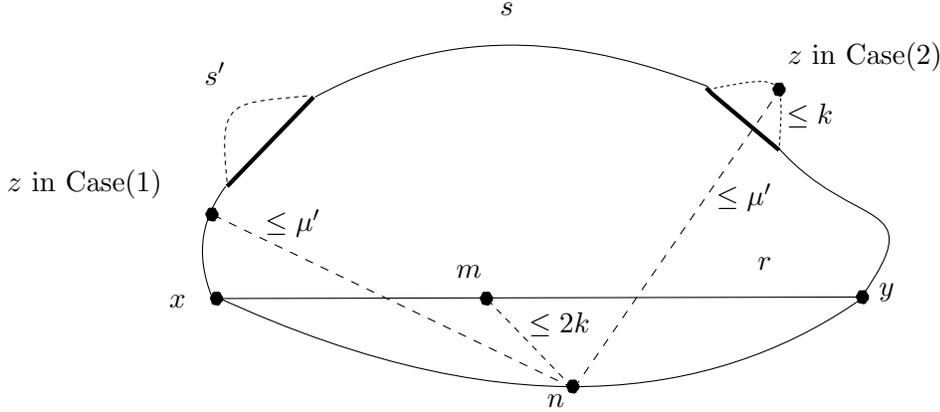
\caption{Corresponding to Theorem \ref{B}}
\label{bowdiag}
\end{figure}

Let $x,y$ be two vertices. Let $m$ be the midpoint of a geodesic $r$ in $\Delta$ connecting them. Let $s$ be any path from $x$ to $y$ in $\Delta$. The path $s$ consists of edges of two types\begin{enumerate}
\item[(i)] edges of the graph $\Sigma$;
\item[(ii)] edges added in transforming $\Sigma$ to $\Delta$ (marked as bold edges on Fig.\ref{bowdiag}). \end{enumerate}
Let $p$ be a geodesic in $\Sigma$ between $x$ and $y$. By Part (a), there exists $k$ such that $p$ is in the $k$-neighborhood of $r$ in $\Delta$. Applying Lemma \ref{neighborhood} , we get a point $n$ on $p$ such that $$d_\Delta(m,n) \leq 2k.$$ 

Let $s'$ be the path in $\Sigma$ between $x$ and $y$, obtained from $s$ by replacing every edge $e$ of type (ii) by a geodesic path $t(e)$ in $\Sigma$ between its end points (marked by dotted lines in Fig.\ref{bowdiag}). Since $\Sigma$ is a quasi-tree, by Lemma \ref{bpr}, there exists $\mu' >0$ and a point $z$ on $s'$ such that $$d_\Sigma(z,n) \leq \mu'.$$
\begin{enumerate}
\item[Case 1:] If $z$ lies on an edge of $s$ of type (i) , then $$d_\Delta(z,m) \leq d_\Delta(z,n) +d_\Delta(n,m) \leq d_\Sigma(z,n) + d_\Delta(n,m) \leq \mu' + 2k.$$
\item[Case 2:]If $z$ lies on a path $t(e)$ that replaced an edge $e$ of type (ii), then by Part (a), $$d_\Delta(e_{-}, m) \leq d_\Delta(e_{-}, z) + d_\Delta(z, n) + d_\Delta(n, m) \leq k+ \mu' + 2k = \mu' +3k. $$
\end{enumerate} 
Thus the bottleneck property holds for $\mu = \mu'+ 3k > 0$.  \end{proof}

\section{Proof of the main result} \label{constr}

Our main result is the following theorem, from which Theorem \ref{main} and other corollaries stated in the introduction can be easily derived (see Section \ref{secappln}). 
\begin{thm} \label{3 condns} Let $\{H_1, H_2,...,H_n\}$ be a finite collection of countable subgroups of a group $G$ such that $\{H_1,H_2,...,H_n\} \hookrightarrow_{h} (G,Z)$ for some $Z\subset G$. Let $K$ be a subgroup of $G$ such that $H_i \leq K$ for all $i$. Then there exists a subset $Y \subset K$ such that: \begin{itemize}
\item[(a)]$\{H_1, H_2,...,H_n\} \hookrightarrow_h (K,Y)$
\item[(b)]$\Gamma(K, Y \sqcup \mathcal{H})$ is a quasi-tree, where $\mathcal{H} = \bigsqcup_{i=1}^{n}H_i$
\item[(c)]The action of $K$ on $\Gamma(K, Y \sqcup \mathcal{H})$ is acylindrical
\item[(d)] $Z \cap K \subset Y$ 
\end{itemize}
\end{thm}

\subsection{Outline of the proof} 
\textbf{Step 1:} In order to prove Theorem \ref{3 condns}, we first prove the following proposition. It is distinct from Theorem \ref{3 condns} since it does not require the action of $K$ on the Cayley graph $\Gamma(K, X \sqcup \mathcal{H})$ to be acylindrical. 
\begin{prop} \label{prop} Let $\{H_1, H_2,...,H_n\}$ be a finite collection of countable subgroups of a group $G$ such that $\{H_1, H_2,...,H_n\} \hookrightarrow_{h} G$ with respect to a relative generating set $Z$. Let $K$ be a subgroup of $G$ such that $H_i \leq K$ for all $i$. Then there exists $X \subset K$ such that \begin{itemize} 
\item[(a)] $\{H_1, H_2,..., H_n\} \hookrightarrow_{h}  (K,X)$ 
\item[(b)] $\Gamma(K, X \sqcup \mathcal{H})$ is a quasi-tree, where $\mathcal{H} = \bigsqcup_{1=1}^{n} \{H_i\}$
\item[(c)] $Z \cap K \subset X$
\end{itemize} 
\end{prop} 

\vspace{10pt} \noindent \textbf{Step 2:} Once we have proved Proposition \ref{prop}, we will utilize an 'acylindrification' construction from \cite{ah} to make the action acylindrical, which will prove Theorem \ref{3 condns}. The details of this step are as follows. 
\begin{proof} By Proposition \ref{prop}, there exists $X \subseteq K$ such that \begin{enumerate} 
\item[(a)] $\{H_1, H_2,...,H_n\} \hookrightarrow_{h}  (K,X)$
\item[(b)] $\Gamma(K, X \sqcup \mathcal{H})$ is a quasi-tree 
\item[(c)] $Z \cap K \subset X$
\end{enumerate} 

By applying Theorem \ref{DO} to the above, we get that there exists $Y \subset K$ such that \begin{enumerate}
\item[(a)] $X \subseteq Y$ 
\item[(b)] $\{H_1,H_2,...,H_n\} \hookrightarrow_{h} (K,Y)$. In particular, the Cayley Graph $\Gamma(K, Y \sqcup \mathcal{H})$ is hyperbolic
\item[(c)] The action of $K$ on $\Gamma (K, Y \sqcup \mathcal{H})$ is acylindrical.\end{enumerate}

From the proof of Theorem \ref{DO} (see \cite{ah} for details), it is easy to see that the Cayley graph $\Gamma(G, Y \sqcup \mathcal{H})$ is obtained from $\Gamma(G, X \sqcup \mathcal{H})$ in a manner that satisfies the assumptions of Theorem \ref{B}, with $M =1$. Thus by Theorem \ref{B}, $\Gamma(K, Y \sqcup \mathcal{H})$ is also a quasi-tree. Further $$K \cap Z \subset X \subset Y.$$ Thus $Y$ is the required relative generating set. \end{proof}

We will thus now focus on proving Proposition \ref{prop}. In order to prove this proposition, will use a construction introduced by Bestvina, Bromberg and Fujiwara in \cite{BBF}. We describe the construction below.

\subsection{The projection complex}\label{setting}
\begin{defn} \label{pc} Let $\mathbb{Y}$ be a set. Suppose that for each $Y \in \mathbb{Y}$ we have a function
\begin{center} $d^{\pi} _Y\colon(\mathbb{Y} \backslash \{Y\} \times \mathbb{Y} \backslash \{Y\}) \rightarrow [0,\infty)$ \end{center}
called a \textit{projection} on $Y$, and a constant $\xi > 0$ that satisfy the following axioms for all $Y$ and all $A,B,C \in \mathbb{Y} \backslash \{Y\}$ :
\begin{enumerate}
\item[(A1)] $d^{\pi} _Y(A,B) = d^{\pi} _Y(B,A)$
\item[(A2)] $d^{\pi} _Y(A,B) + d^{\pi} _Y(B,C) \geq d^{\pi} _Y(A,C)$
\item[(A3)] min $\{ d^{\pi} _Y(A,B), d^{\pi} _B(A,Y) \} < \xi $
\item[(A4)] $\# \{ Y \mid d^{\pi} _Y (A,B) \geq \xi \}$ is finite 
\end{enumerate}
\end{defn}

Let $J$ be a positive constant. Then associated to this data we have the \textit{projection complex} $P_J(\mathbb{Y})$, which is a graph constructed in the following manner $\colon$ the set of vertices of  $P_J(\mathbb{Y})$ is the set $\mathbb{Y}$. To specify the set of edges, one first defines a new function $d _Y \colon(\mathbb{Y} \backslash \{Y\} \times \mathbb{Y} \backslash \{Y\}) \rightarrow [0,\infty)$, which can be thought of as a small perturbation of $d^{\pi} _Y$. The exact definition of $d_Y$ can be found in \cite{BBF}. An essential property of the new function is the following inequality, which is an immediate corollary of \cite{BBF}, Proposition 3.2. 

For every $Y \in \mathbb{Y}$ and every $A,B \in \mathbb{Y} \backslash \{Y\}$, we have $$ | d^{\pi} _Y(A,B) - d _Y(A,B) | \leq 2\xi.  \hspace{40pt} (1) $$ 

The set of edge of the graph $P_J(\mathbb{Y})$ can now be described as follows $\colon$ two vertices $ A, B \in \mathbb{Y}$ are connected by an edge if and only if for every $Y \in \mathbb{Y} \backslash \{A, B\}$, $d _Y(A, B) \leq J$. This construction strongly depends on the constant $J$. Complexes corresponding to different $J$ are not isometric in general. 

We would like to mention that if $\mathbb{Y}$ is endowed with an action of a group $G$ that preserves projections, i.e., $d^{\pi}_{g(Y)} (g(A), g(B)) = d^{\pi}_Y(A, B)$ ), then the action of $G$ can be extended to an action on $P_J(\mathbb{Y})$.  We also mention the following proposition, which has been proved under the assumptions of Definition \ref{pc}.

\begin{prop}[\cite{BBF}, Theorem 3.16] \label{K} For a sufficiently large $J > 0$, $P_J(\mathbb{Y})$ is connected and quasi-isometric to a tree.
\end{prop}

\begin{defn}\label{proj} [Nearest point projection] In a metric space $(S,d)$, given a set $Y$ and a point $a \in S$, we define the nearest point projection as $$proj_Y(a) = \{ y \in Y \mid d(Y, a) =d(y, a) \}.$$ 

If A, Y are two sets in S, then $$proj_Y(A) = \bigcup_{a \in A} proj_Y(a).$$ 
\end{defn}

We note that in our case, since elements of $\mathbb{Y}$ will come from a Cayley graph, which is a combinatorial graph, the nearest point projection will exist. This is because distances on a combinatorial graph take discrete values in $\mathbb{N} \cup \{0\}$. Since this set is bounded below, we cannot have an infinite strictly decreasing sequence of distances. This may not be the case if we have a real tree. For example, on the real line, the point $0$ has no nearest point projection to the open interval $(0,1)$. 

We make all geometric considerations in the Cayley graph $\Gamma(G, Z \sqcup \mathcal{H})$. Let $\dzh$ denote the metric on this graph. Since $\{H_1,H_2,...,H_n\} \hookrightarrow_h G$ under the assumptions of Proposition \ref{prop}, by Remark 4.26 of \cite{DGO}, $H_i \hookrightarrow_h G$ for all $i =1,2,...,n$. By Theorem \ref{tilde}, we can define a finite valued, proper metric $\widetilde{d_i}$ on $H_i$, for all $i=1,2,...,n$, satisfying \begin{center}$\widetilde{d_i}(x, y) \leq \widehat{d_i}(x,y)$ for all $x,y \in H_i$  and for all $i=1,2,...,n$ \hspace{10pt} (2)\end{center}

We can extend both $\widehat{d_i}$ and $\widetilde{d_i}$ to all cosets $gH_i$ of $H_i$ by setting $\widetilde{d_i}(gx, gy) = \widetilde{d_i}(x, y)$ and $\widehat{d_i}(gx, gy) = \widehat{d_i}(x, y)$ for all $x, y \in H_i$. Let $\widehat{diam}$ (resp. $\widetilde{diam}$) denote the diameter of a subset of $H_i$ or a coset of $H_i$ with respect to the $\widehat{d_i}$ (resp. $\widetilde{d_i}$) metric.\\

Let $$\mathbb{Y} = \{ k H_i \mid k \in K, i=1,2,...,n\}$$ be the set of cosets of all $H_i$ in K. We think of cosets of $H_i$ as a subset of vertices of $\Gamma(G, Z \sqcup \mathcal{H})$. \\

For each $Y \in \mathbb{Y}$, and $A, B \in \mathbb{Y} \backslash \{Y \}$, define
$$d^{\pi}_Y (A, B) = \widetilde{diam}(proj_Y(A) \cup proj_Y(B) ), \hspace{30pt}  (3)$$ where $proj_Y(A)$ is defined as in Definition \ref{proj}. The fact that (3) is well-defined will follow from Lemma \ref{oneptbdd} and Lemma \ref{bdd}, which are proved below. We will also proceed to verify the axioms $(A1)-(A4)$ of the Bestvina-Bromberg-Fujiwara construction in the above setting. 

\begin{figure}
  \centering
  \def\svgscale{0.25}
  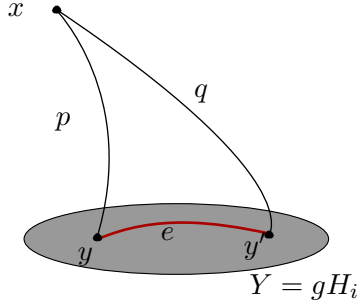
\caption{The bold red edge denotes a single edge labelled by an element of $\mathcal{H}$}
\label{ptbdd}
\end{figure}

\begin{lem} \label{oneptbdd} For any $Y \in \mathbb{Y}$ and any $x \in G$, $\widetilde{diam}(proj_Y(x))$ is bounded. \end{lem}
\begin{proof} By (2), it suffices to prove that $\widehat{diam}(proj_Y(x))$ is bounded. Let $y, y' \in proj_Y(x)$. Then $\dzh(x, y) = \dzh(x, y') = \dzh(x, Y)$.  Without loss of generality, $x \notin Y$, else the diameter is zero.

Let $Y = gH_i$. Let $e$ denote the edge connecting $y$ and $y'$, and labelled by an element of $H_i$. Let $p$ and $q$ denote the geodesics between the points $x$ and $y$, and $x$ and $y'$ respectively. (see Fig.\ref{ptbdd})

Consider the geodesic triangle $T$ with sides $e, p, q$. Since $p$ and $q$ are geodesics between the point $x$ and $Y$, $e$ is an isolated component in $T$, i.e., $e$ cannot be connected to either $p$ or $q$. Indeed if $e$ is connected to, say, a component of $p$, then that would imply that $e_+$ and $e_{-}$ are in $Y$,  i.e., the geodesic $p$ passes through a point of $Y$ before $y$. But then $y$ is not the nearest point from $Y$ to $x$, which is a contradiction. By Lemma \ref{isolated}, $\widehat{d}_i(y, y') \leq 3C$. Hence $$\widehat{diam}(proj_Y(x)) \leq 3C.$$ \end{proof}

\begin{lem} \label{bdd} For every pair of distinct elements $A, Y \in \mathbb{Y}$, $\widehat{diam}(proj_Y(A)) \leq 4C$, where $C$ is the constant as in Lemma \ref{isolated}. As a consequence,  $\widetilde{diam}(proj_Y(A))$ is bounded. \end{lem}
\begin{proof} let $Y=gH_i$ and $A= fH_j$. Let $y_1, y_2 \in proj_Y(A)$. Then there exist $a_1, a_2 \in A$ such that $\dzh(a_1,y_1) =\dzh(a_1,Y)$ and $\dzh(a_2, y_2) = \dzh(a_2, Y)$. Now $y_1$ and $y_2$ are connected by a single edge $e$, labelled by an element of $H_i$, and similarly, $a_1$ and $a_2$ are connected by an edge $f$, labelled by an element of $H_j$ (see Fig.\ref{wd}). Let $p$ and $q$ denote geodesics that connect $y_1 , a_1$ and $y_2 ,a_2$ respectively. We note that $p$ and/or $q$ may be trivial paths (consisting of a single point), but this does not alter the proof.

\begin{figure}
  \centering
  \def\svgscale{0.35}
  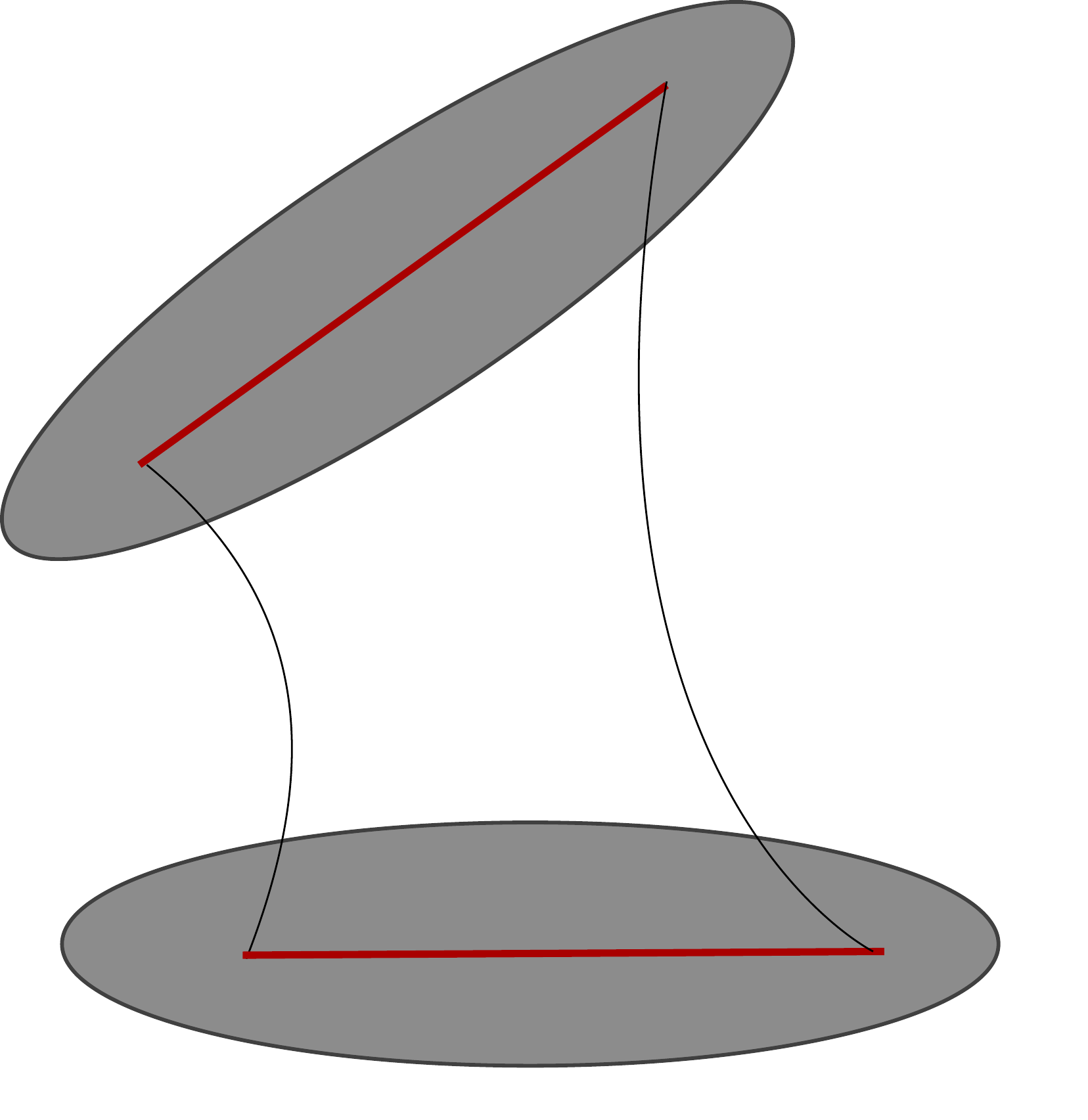
\caption{Lemma \ref{bdd}}
\label{wd}
\end{figure}

Consider $e$  in the quadrilateral $Q$ with sides $p, f, q, e$. As argued in the previous lemma, since $p$ (respectively $q$) is a path of minimal length between the points $a_1$ (respectively $a_2$) and Y, $e$ cannot be connected to a component of $p$ or $q$.

If $i = j$, then $e$ cannot be connected to $f$ since $A \neq Y$. If $i \neq j$, then obviously $e$ and $f$ cannot be connected. Thus $e$ is isolated in this quadrilateral $Q$. By Lemma \ref{isolated}, $\widehat{d}_i(y_1, y_2) \leq 4C$. Thus $$\widehat{diam}(proj_Y(A)) \leq 4C.$$\end{proof}

\begin{cor} The function $d^{\pi}_Y$ defined by (3) is well-defined. \end{cor}
\begin{proof} Since the $\widetilde{d}_i$ metric takes finite values for all $i=1,2,...,n$, using Lemma \ref{bdd}, we have that $d^{\pi}_Y$  also takes only finite values. \end{proof}

\begin{lem} \label{a1a2} The function $d^{\pi}_Y$ defined by (3) satisfies conditions (A1) and (A2) in Definition \ref{pc}\end{lem}
\begin{proof} (A1) is obviously satisfied. For any $Y \in \mathbb{Y}$ and any $A,B,C \in \mathbb{Y} \backslash \{Y \}$, by the triangle inequality, we have that $$d^{\pi}_Y(A,C) = \widetilde{diam}(proj_Y(A) \cup proj_Y(C))$$  $$\leq \widetilde{diam}(proj_Y(A) \cup proj_Y(B)) + \widetilde{diam}(proj_Y(B) \cup proj_Y(C))$$  $$ =d^{\pi}_Y(A, B) + d^{\pi}_Y(B, C).$$  Thus (A2) also holds. \end{proof}

\begin{lem} \label{a3} The function $d^{\pi}_Y$ from (3) satisfies condition (A3) in Definition \ref{pc} for any $\xi >14C$, where $C$ is the constant from Lemma \ref{isolated}\end{lem}
\begin{proof} By (2), it suffices to prove that $$ min\{ \widehat{diam}(proj_Y(A) \cup proj_Y(B)), \widehat{diam}(proj_B(A) \cup proj_B(Y)) \} < \xi.$$

 Let $A, B \in \mathbb{Y} \backslash \{Y\}$ be distinct. Let $Y =gH_i$, $A =fH_j$ and $B= tH_k$. If $\widehat{diam}(proj_Y(A) \cup proj_Y(B)) \leq 14C$, then we are done. So let $$\widehat{diam}(proj_Y(A) \cup proj_Y(B))) > 14C. \hspace{25pt} (4)$$

\begin{figure}
\centering
\def\svgscale{0.5}
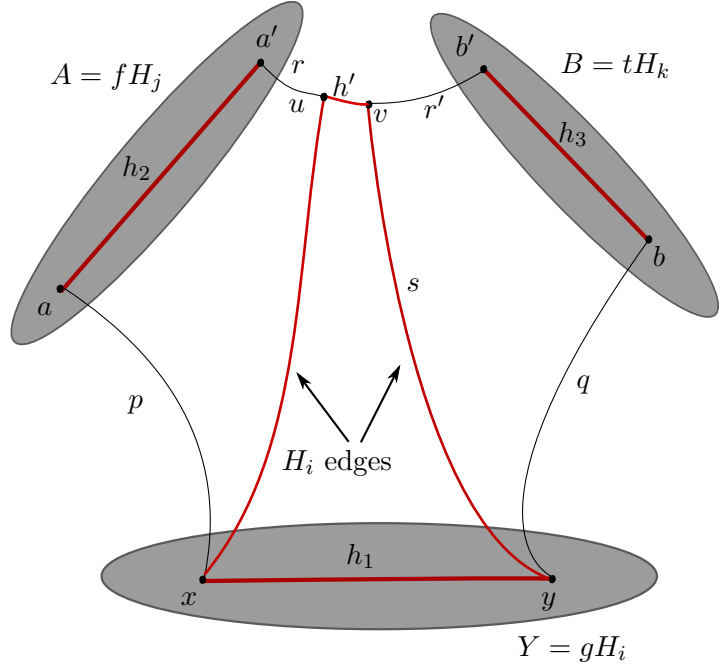
\caption{Condition $(A3)$}
\label{a3fig}
\end{figure}

Choose $a \in A, b \in B$, and $x,y \in Y$ such that $\dzh(A, Y) = \dzh(a,x)$ and $\dzh(B, Y) = \dzh(b,y)$. In particular, $$x \in proj_Y(A),y \in proj_Y(B) \hspace{25pt} (5)$$ and $b \in proj_B(Y)$. Let $p, q$ denote the geodesics connecting $a,x$ and $b, y$ respectively. Let $h_1$ denote the edge connecting $x$ and $y$, which is labelled by an element of $H_i$.

By (5), we have that $$\widehat{diam}(proj_Y(A) \cup proj_Y(B))  \leq \widehat{diam}(proj_Y(A)) + \widehat{diam}(proj_Y(B)) + \widehat{d}_i(x,y).$$ Combining this with (4) and Lemma \ref{bdd}, we get $$\widehat{d}_i(x,y) \geq \widehat{diam}(proj_Y(A) \cup proj_Y(B)) - \widehat{diam}(proj_Y(A))  -  \widehat{diam}(proj_Y(B))$$  $$> 14C - 8C = 6C.$$ 

\begin{figure}
  \centering
  \def\svgscale{0.325}
  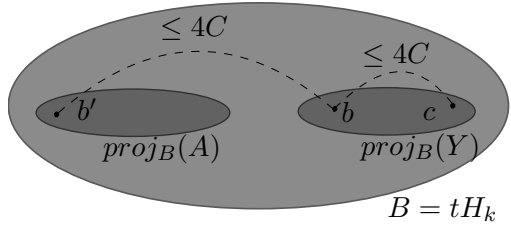
\caption{Estimating the distance between arbitrary points $b$ and $c$ of $proj_B(A)$ and $proj_B(Y)$ resp.}
\label{dist}
\end{figure}

Choose any $a' \in A$ and $b' \in proj_B(a')$, i.e.,: $\dzh(a', B) =\dzh(a', b')$; (see Fig.\ref{a3fig}). (Note that if $a'=a$, the following arguments still hold). Let $h_2$ and $h_3$ denote the edges connecting $a,a'$ and $b, b'$; which are labelled by elements of $H_j$ and $H_k$ respectively. Let $r$ denote the geodesic connecting $a'$ and $b'$. Consider the geodesic hexagon $W$ with sides $p, h_1, q, h_3, r, h_2$. Then $h_1$ is not isolated in $W$, else by Lemma \ref{isolated}, $\widehat{d}_i(x,y) \leq 6C$, a contradiction.

Thus $h_1$ is connected to another $H_i$-component in $W$. Arguing as in Lemma \ref{bdd}, $h_1$ cannot be connected to a component of $p$ or $q$. Since $A,B,Y$ are all distinct, $h_1$ cannot be connected to $h_2$ or $h_3$. So $h_1$ must be connected to an $H_i$-component on the geodesic $r$. Let this edge be $h'$ with end points $u$ and $v$ as shown in Fig \ref{a3fig}. Let $s$ denote the edge (labeled by an element of $H_i$), that connects $y, v$. Let $r'$ denote the segment of $r$ that connects $v$ to $b'$. Then $r'$ is also a geodesic.

Consider the quadrilateral $Q$ with sides $r', h_3, q, s$. As argued before, $h_3$ cannot be connected to $r', q$ or $s$. Thus $h_3$ is isolated in $Q$. By Lemma \ref{isolated}, $$\widehat{d}_k(b, b') \leq 4C.$$

Since the above argument holds for any $a' \in A$ and for $b' \in proj_B(A)$, we have that $\widehat{d_k}(b, b') \leq 4C$. Using Lemma \ref{bdd} (see Fig.\ref{dist}), we get that $$\widehat{diam}(proj_B(Y) \cup  proj_B(A))  \leq 4C + 4C = 8C < \xi.$$  \end{proof}

\begin{figure}
\centering
\def\svgscale{0.45}
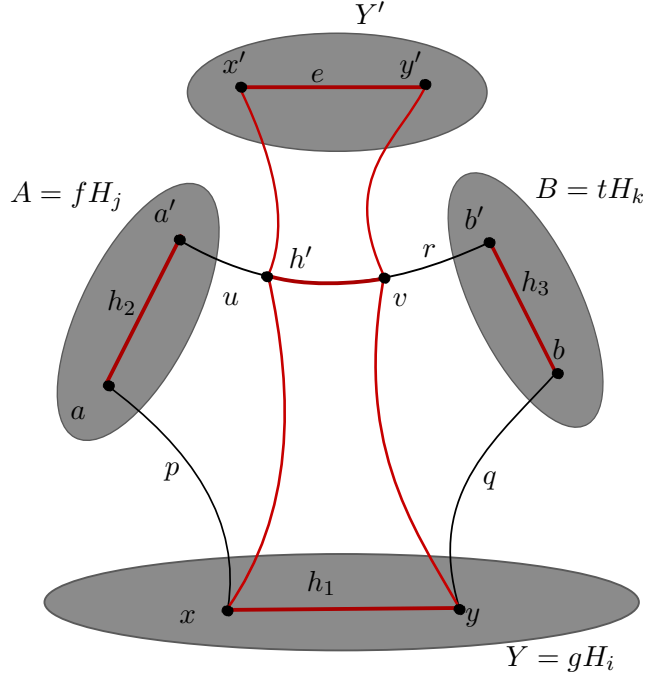 
\caption{Condition $(A4)$}
\label{a4fig}
\end{figure}

\begin{lem} \label{a4} The function $d^{\pi}_Y$ defined by (3)  satisfies condition (A4) in Definition \ref{pc}, for $\xi > 14C$, where $C$ is the constant from Lemma \ref{isolated}\end{lem}
\begin{proof}If $d_{Y}^{\pi}(A,B) \geq \xi$, then by (2), $\widehat{diam}(proj_Y(A) \cup proj_Y(B)) \geq  d_{Y}^{\pi}(A,B) \geq \xi$. Thus it suffices to prove that the number of elements $Y \in \mathbb{Y}$ satisfying $$ \widehat{diam}(proj_Y(A) \cup proj_Y(B)) \geq \xi \hspace{25pt} (6)$$ is finite. Let  $A,B\in \mathbb{Y}, A= fH_j$ and $B=tH_k$. Let $Y \in \mathbb{Y} \backslash \{A,B\}, Y= gH_i$. Let $a' \in A, b' \in proj_B(a')$. Arguing as in Lemma \ref{a3}, if  is such that $\widehat{diam}(proj_Y(A) \cup proj_Y(B)) \geq \xi$, then for any $a \in A, b \in B, x \in proj_Y(a), y \in proj_Y(b)$, we have that $\widehat{d}_i(x,y) > 6C$.

Let $h_1$ denote the edge connecting $x, y$, which is labelled by an element of $H_i$ (see Fig.\ref{a4fig}). Let $h_2$  denote the edge connecting $a, a'$, which is labelled by an element of $H_j$ and $h_3$  denote the edge connecting $b, b'$, which is labelled by an element of $H_k$. Let $p$ be a geodesic between $a, x$, let $q$ be a geodesic between $b, y$, and let $r$ be a geodesic between $a', b'$. As argued in Lemma \ref{a3}, we can show that $h_1$ cannot be isolated in the hexagon $W$ with sides $p, h_1, q, h_2, r, h_3$ and must be connected to an $H_i$-component of $r$, say the edge $h'$.

We claim that the edge $h'$ uniquely identifies $Y$. Indeed, let $Y'$ be a member of $\mathbb{Y}$, with  elements $x',y'$ connected by an edge $e$ (labelled by an element of the corresponding subgroup). Suppose that $e$ is connected to $h'$. Then we must have that $Y'$ is also a coset of $H_i$. But cosets of a subgroup are either disjoint or equal, so $Y =Y'$. Thus, the number of $Y \in \mathbb{Y}$ satisfying (6) is bounded by the number of distinct $H_i$-components of $r$, which is finite. \end{proof}

\subsection{Choosing a relative generating set} \label{X}
We now have the necessary details to choose a relative generating set $X$ which will satisfy conditions (a) and (b) of Proposition \ref{prop}. This set will later be altered slightly to obtain another relative generating set which will satisfy all three conditions of Proposition \ref{prop}. We will repeat arguments similar to those from pages 60-63 of \cite{DGO}. 

Recall that $\mathcal{H} = \bigsqcup_{i =1}^{n} H_i$, and $Z$ is the relative generating set such that $\{H_1, H_2,...,H_n\} \hookrightarrow_h (G,Z)$. Let $P_J(\mathbb{Y})$ be the projection complex corresponding to the vertex set $\mathbb{Y}$ as specified in section \ref{setting} and the constant $J$ is as in Proposition \ref{K}, i.e.,  $P_J(\mathbb{Y})$ is connected and a quasi-tree. Let $d_P$ denote the combinatorial metric on$P_J(\mathbb{Y})$. Our definition of projections is $K$- equivariant and hence the action of $K$ on $\mathbb{Y}$  extends to a cobounded action of $K$ on $P_J(\mathbb{Y})$.

In what follows, by considering $H_i$ to be vertices of the projection complex $P_J(\mathbb{Y})$, we denote by $star(H_i)$, the set  $$\{kH_j \in \mathbb{Y} \left| d_P(H_i, kH_j) =1 \right. \}.$$

We choose the set $X$ in the following manner. For all $i=1,2,...,n$ and each edge $e$ in $star(H_i)$ in  $P_J(\mathbb{Y})$ that connects $H_i$ to $kH_j$, choose an element $x_e \in H_i k H_j$ such that $$\dzh(1, x_e) = \dzh(1, H_i k H_j).$$ 

We say that such an $x_e$ has \textit{type $(i, j)$}. Since $H_i \leq K$ for all $i$, $x_e \in K$. 
We observe the following:\begin{itemize}
\item[(a)] For each $x_e$ of type $(i,j)$ as above, there is an edge in  $P_J(\mathbb{Y})$ connecting $H_i$ and $x_eH_j$. Indeed if $x_e = h_1kh_2$, for $h_1 \in H_i, h_2 \in H_j$,  then $$d_P(H_i, x_eH_j) = d_P(H_i, h_1kh_2H_j) = d_P(H_i, h_1kH_j)$$  $$= d_P(h_1 ^{-1}H_i, kH_j) = d_P(H_i, kH_j) = 1.$$

\item[(b)] For each edge $e$ connecting $H_i$ and $kH_j$, there is a  dual edge $f$ connecting $H_j$ and $k^{-1}H_i$. We will  choose the elements $x_e$ and $x_f$ to be mutually inverse. In particular, the set given by $$X = \{ x_e \neq 1 | e\in star(H_i), i=1,2,...,n \} \hspace{25pt} (7)$$ is symmetric, i.e., closed under taking inverses. Obviously, $X \subset K$. 

\item[(c)] If $x_e \in X$ is of type $(i, j)$, then $x_e$ is not an element of $H_i$ or $H_j$. Indeed if $x_e =h_1kh_2 \in H_i$ for some $h_1 \in H_i$ and some $h_2 \in H_j$, then $ k= hf$ for some $h \in H_i$ and some $f \in H_j$. Consequently $$\dzh(1, H_i k H_j) = \dzh(1, H_i H_j) = 0 = \dzh(1, x_e),$$ which implies $x_e =1$, which is a contradiction to (7).
\end{itemize} 

\begin{lem}[cf. Lemma 4.49 in \cite{DGO}] The subgroup $K$ is generated by $X$ together with the union of all $H_i$'s. Further, the Cayley graph $\Gamma(K, X \sqcup \mathcal{H})$ is quasi-isometric to  $P_J(\mathbb{Y})$, and hence a quasi-tree. \end{lem}
\begin{proof} Let $\Sigma = \{H_1, H_2,...,H_n\} \subseteq \mathbb{Y}$. Let $diam(\Sigma)$ denote the diameter of the set $\Sigma$ in the combinatorial metric $d_P$. Since $\Sigma$ is a finite set, $diam(\Sigma)$ is finite. Define \begin{center} $\phi\colon K \rightarrow \mathbb{Y}$ as $\phi(k) = k H_1$ \end{center} By Property (a) above, if $x_e \in X$ is of type $(i, j)$, $$d_P(x_eH_1, H_1) \leq d_P(x_eH_1, x_e H_j) + d_P(x_eH_j, H_i) + d_P(H_i,H_1)$$  $$= d_P(H_1, H_j) + 1 + d_P(H_i, H_1) \leq 2 diam(\Sigma) +1.$$ Further, for $h \in H_i$, $$d_P(h H_1, H_1) \leq d_P(hH_1,h H_i) + d_P(h H_i, H_1)$$ $$= d_P(H_1, H_i) + d_P(H_i, H_1) \leq 2 diam(\Sigma).$$ Thus for all $g \in \langle X \cup H_1 \cup H_2...\cup H_n \rangle$, we have $$d_P(\phi(1), \phi(g)) \leq (2diam(\Sigma) +1) |g|_{X \sqcup \mathcal{H}}, \hspace{25pt} (8)$$ where $|g|_{X \sqcup \mathcal{H}}$ denotes the length of $g$ in the generating set $X \cup H_1 \cup H_2...\cup H_n$. (We use this notation for the sake of uniformity). 

Now let $g \in K$ and suppose $d_P(\phi(1), \phi(g)) =r$, i.e., $d_P(H_1, gH_1) = r$. If $r =0$, then $H_1 =gH_1$, thus $g \in H_1$ and $|g|_{X \sqcup \mathcal{H}} \leq 1$. If $ r >0$, consider the geodesic $p$ in  $P_J(\mathbb{Y})$ connecting $H_1$ and $gH_1$. Let $$v_0 = H_1 = g_0 H_1 (g_0 =1), v_1 = g_1 H_{\lambda_1}, v_2 = g_2 H_{\lambda_2},..., v_{r-1} = g_{r-1} H_{\lambda_{r-1}},  v_r = gH_1 $$ be the sequence of vertices of $p$, for some $\lambda_j \in \{1,2,...,n\}$, and some $g_i \in K$ (see Fig.\ref{p}).

\begin{figure}
\centering
\def\svgscale{0.5}
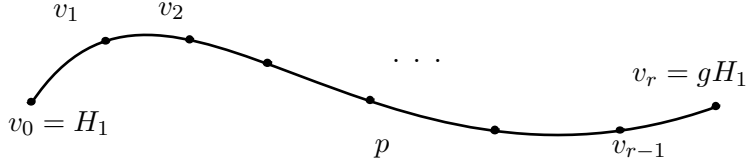
\caption{The geodesic $p$}
\label{p}
\end{figure}

Now $g_iH_{\lambda_i}$ is connected by a single edge to $g_{i+1}H_{\lambda_{i+1}}$. Thus $d_P(g_iH_{\lambda_i}, g_{i+1}H_{\lambda_{i+1}}) =1$, which implies $d_P(H_{\lambda_i}, g_i ^{-1} g_{i+1} H_{\lambda_{i+1}}) =1$. Then there exists $x \in X$ such that $$x \in H_{\lambda_i} g_i ^{-1} g_{i+1} H_{\lambda_{i+1}}$$ and $$\dzh(1, x) = \dzh(1, H_{\lambda_i} g_i ^{-1} g_{i+1} H_{\lambda_{i+1}}).$$ Thus $ x = h g_i ^{-1} g_{i+1} k $ for some $h \in H_{\lambda_i}$ and some $k \in H_{\lambda_{i+1}}$ which implies $g_i ^{-1} g_{i+1} = h^{-1} x k^{-1}$. So $|g_i ^{-1} g_{i+1}|_{X \sqcup \mathcal{H}} \leq 3$, which implies $$|g|_{X \sqcup \mathcal{H}} = \left| \prod _{i=1} ^{r} g_{i-1}^{-1} g_i \right|_{X \sqcup \mathcal{H}} \leq \sum_{i=1} ^{r}\left| g_{i-1}^{-1} g_i \right|_{X \sqcup \mathcal{H}} \leq 3r = 3 d_P(\phi(1), \phi(g)) \hspace{20pt}(9)$$

The above argument also provides a representation for every element $g \in K$ as a product of elements from $X \cup H_1 \cup H_2...\cup H_n$. Thus $K$ is generated by the union of $X$ and all $H_i$'s. By (8) and (9), $\phi$ is a quasi-isometric embedding of $(K, |.|_{X \sqcup \mathcal{H}})$ into $(P_J(\mathbb{Y}), d_P)$ satisfying $$\frac{1}{3}|g|_{X \sqcup \mathcal{H}} \leq d_P(\phi(1),\phi(g)) \leq (2diam(\Sigma) +1) |g|_{X \sqcup \mathcal{H}}.$$ Since $\mathbb{Y}$ is contained in the closed $diam(\Sigma)$-neighborhood of $\phi(K)$, $\phi$ is a quasi-isometry. This implies that $\Gamma (K, X \sqcup \mathcal{H})$  is a quasi-tree.  \end{proof}

Let $\widetilde{d_i}$ denote the modified relative metric on $H_i$ associated with the Cayley graph $\Gamma(G, Z \sqcup \mathcal{H})$ from Theorem \ref{tilde}. Let $\widehat{d_{i}^{X}}$ denote the relative metric on $H_i$ associated with the Cayley graph $\Gamma(K, X \sqcup \mathcal{H}).$ We will now show that $\widehat{d_{i}^{X}}$ is proper for all $ i=1,2,...,n$. We will use the fact that $\widetilde{d_i}$ is proper and derive a relation between $\widetilde{d_i}$ and $\widehat{d_{i}^{X}}$.

\begin{lem}[cf. Lemma 4.50 in \cite{DGO}]\label{alpha} There exists a constant $\alpha$ such that for any $Y \in \mathbb{Y}$ and any $x \in X \sqcup \mathcal{H}$,  if $$\widetilde{diam}(proj_Y\{1,x\}) > \alpha,$$ then $x \in H_j$ and $Y =H_j$ for some $j$. \end{lem}

\begin{proof} We prove the result for $$\alpha = max \{ J+2\xi, 6C \}.$$ Suppose that $\widetilde{diam}(proj_Y \{1,x\}) > \alpha$ and $x \in X$ has type $(k, l)$, i.e., there exists an edge connecting $H_k$ and $gH_l$ in $P_J(\mathbb{Y})$, where $g \in K$. We consider three possible cases and arrive at a contradiction in each case.
\begin{itemize}
\item[Case 1:] $H_k \neq Y \neq xH_l$. Then $$\widetilde{diam}(proj_Y \{1,x\}) \leq d^{\pi} _Y(H_k, xH_l) \leq d_Y(H_k, xH_l) + 2\xi \leq J + 2\xi \leq \alpha,$$  using (1) and the fact that $H_k$ and $xH_l$ are connected by an edge in $P_J(\mathbb{Y})$, which is a contradiction.

\item[Case 2:] $H_k =Y$. Since $x \notin H_k$, let $y \in proj_Y(x)$, i.e., $\dzh(x, y) = \dzh(x, H_k) = \dzh(x, Y)$.
\begin{figure}[H]
\centering
\def\svgscale{0.25}
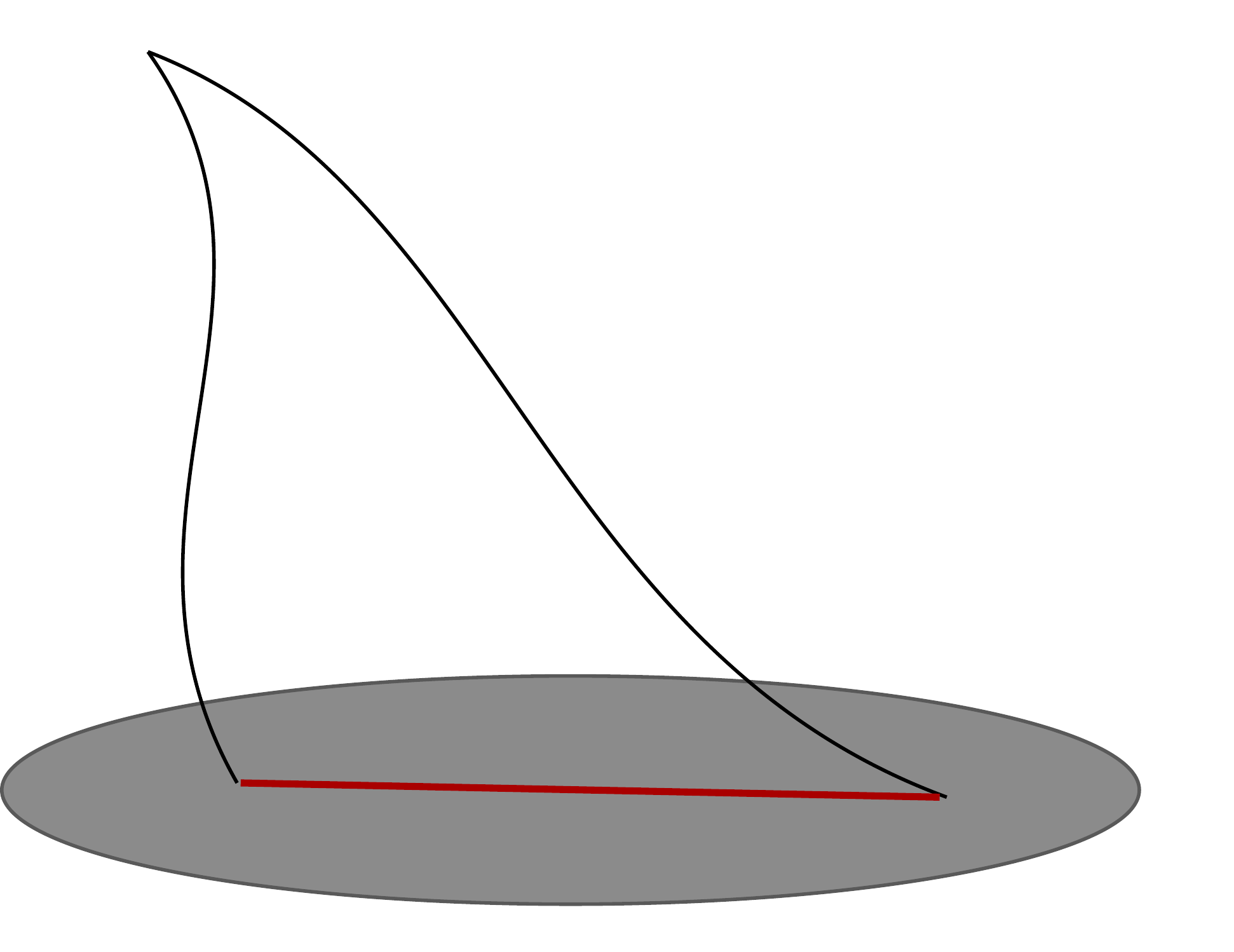
\caption{Case 2}
\label{case2}
\end{figure}

By Lemma \ref{oneptbdd}, if $\widehat{d}_k(1, y) \leq 3C$, then $$\widehat{diam}(proj_Y \{1,x\}) \leq \widehat{diam}(proj_Y(1)) + \widehat{diam}(proj_Y (x)) + \widehat{d}_k( proj_Y(1), proj_Y(x))$$  $$\leq 0 + 3C + \widehat{d}_k(1, y) \leq 6C \leq \alpha.$$ Then by (2), we have $$\widetilde{diam}(proj_Y \{1,x\})\leq \alpha,$$ which is a contradiction. Thus $\widehat{d}_k(1, y) > 3C$. This implies that $ 1 \notin proj_Y(x)$ (see Fig.\ref{case2}). By definition of the nearest point projection, $\dzh(1, x) > \dzh(y, x),$ which implies $\dzh(1, x) > \dzh( 1, y^{-1}x)$. Since $y^{-1}x \in H_k g H_l$, we obtain $\dzh(1, x) > \dzh(1, H_k g H_l)$, which is a contradiction to the choice of $x$.

\item[Case 3:] $Y = xH_l, H_k \neq Y$. This case reduces to Case 2, since we can translate everything by $x^{-1}$.
\end{itemize}
Thus we must have $x \in H_j$ for some $j$. Suppose that $H_j \neq Y$. But then $ \widetilde{diam}(proj_Y \{1,x\}) \leq \widetilde{diam}(proj_Y(H_j)) \leq 4C \leq \alpha$, by Lemma \ref{bdd}; which is a contradiction. \end{proof}

\begin{lem}[cf. Lemma 4.45 in \cite{DGO}]\label{i=j} If $H_i = f H_j$, then $H_i =H_j$ and $f \in H_i$. Consequently, if $g H_i = f H_j$, then $H_i = H_j$ and $g^{-1}f \in H_i$. \end{lem}

\begin{proof} If $H_i = f H_j$, then $1 = fk$ for some $k \in H_j$. Then $ f= k^{-1} \in H_j$, which implies $H_i =H_j$.\end{proof}

\begin{lem}[cf. Theorem 4.42 in \cite{DGO}] For all $i=1,2,...,n$ and any $h \in H_i$, we have $$\alpha \widehat{d^{X}_{i}}(1, h) \geq \widetilde{d_i}(1,h),$$ where $\alpha$ is the constant from Lemma \ref{alpha}. Thus $\widehat{d_{i}^{X}}$ is proper. \end{lem}

\begin{proof} Let $h \in H_i$ such that $\widehat{d^{X}_{i}}(1, h) = r$. Let $e$ denote the $H_i$-edge in the Cayley graph $\Gamma(K, X \sqcup \mathcal{H})$ connecting $h$ to $1$, labeled by $h^{-1}$. Let $p$ be an admissible (see Definition \ref{rm}) geodesic path of length $r$ in $\Gamma(K, X \sqcup \mathcal{H})$ connecting $1$ and $h$. Then $ep$ forms a cycle. Since $p$ is admissible, $e$ is isolated in this cycle.

Let $Lab(p) = x_1 x_2 ...x_r$ for some $x_1, x_2,...,x_r\in X \sqcup \mathcal{H}$. Let $$v_0 =1, v_1 =x_1, v_2 =x_1 x_2 ,..., v_r = x_1 x_2 ...x_r = h.$$ Since these are also elements of $G$, for all $k = 1,2,...,r$ we have $$\widetilde{diam}(proj_{H_i} \{v_{k-1}, v_k\}) = \widetilde{diam}(proj_{H_i} \{x_1 x_2...x_{k-1} , x_1 x_2...x_{k-1} x_k\})$$  $$= \widetilde{diam}(proj_Y\{1, x_k\},$$ where $Y =(x_1 x_2...x_{k-1})^{-1}H_i$.
\begin{figure}
\centering
\def\svgscale{0.4}
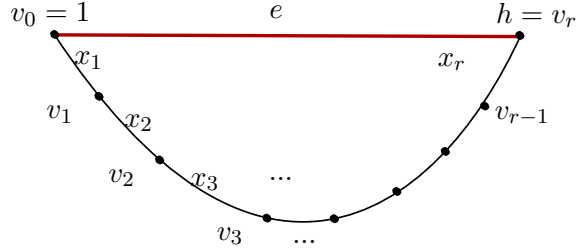
\caption{The cycle $ep$}
\label{cycle}
\end{figure}

If $\widetilde{diam}(proj_Y\{ 1, x_k\}) > \alpha$ for some $k$, then by Lemma \ref{alpha}, $x_k \in H_j$ and $Y =H_j$ for some j. By Lemma \ref{i=j}, $H_i =H_j$ and $x_1 x_2 ...x_{k-1} \in H_j$. But then $e$ is not isolated in the cycle $ep$, which is a contradiction.

Hence $$\widetilde{diam}(proj_{H_i}\{ v_{k-1}, v_k\}) \leq \alpha$$ for all $k = 1, 2,...,r$, which implies $$\widetilde{d_{i}}(1, h) \leq \widetilde{diam}(proj_{H_i}\{v_0, v_r\}) \leq \sum_{j=1}^{r} \widetilde{diam}(proj_{H_i} \{v_{j-1}, v_j\} \leq r\alpha = \alpha \widehat{d^X}_i(1, h).$$
 \end{proof}

\subsection{Proof of Proposition \ref{prop}} \label{rel}

The goal of this section is to alter our relative generating set $X$ from Section \ref{X}, so that we obtain another relative generating set that satisfies all the conditions of Proposition \ref{prop}. To do so, we need to establish a relation between the set $X$ and the set $Z$. We will need the following obvious lemma.

\begin{lem} \label{qt} Let X and Y be generating sets of G, and $sup_{x \in X} |x|_Y < \infty$ and $sup_{y \in Y} |y|_X < \infty$. Then $\Gamma(G, X)$ is quasi-isometric to $\Gamma(G, Y)$. In particular $\Gamma(G, X)$ is a quasi-tree if and only if $\Gamma(G, Y)$ is a quasi-tree.
\end{lem}

\begin{rem}\label{change} The above lemma implies that if we change a generating set by adding finitely many elements, then the property that the Cayley graph is a quasi-tree still holds. \end{rem}

We also need to note that from (1) in Definition \ref{pc}, it easily follows that $$d_Y(A, B) \leq d^{\pi}_Y(A,B) +2\xi. \hspace{30pt} (10)$$

\begin{lem}\label{notrep} For a large enough $J$, the set $X$ constructed in Section \ref{X} satisfies the following property$\colon$ If $z \in Z \cap K$ does not represent any element of $H_i$ for all $i=1,2,...,n$, then $z \in X$. \end{lem}
\begin{proof} Recall that $\dzh$ denotes the combinatorial metric on $\Gamma(G, Z \sqcup \mathcal{H})$. Let $z \in Z \cap K$ be as in the statement of the lemma. Then $z \in H_i z H_i$ for all $i$ and $1 \notin H_i z H_i$. Thus $$\dzh(1, H_i z H_i) \geq 1 = \dzh(1, z) \geq \dzh(1, H_i z H_i),$$ which implies \begin{center}$\dzh(1, H_i z H_i) = \dzh(1, z)$ for all $i$.\end{center}

In order to prove $z \in X$, we must show that $H_i$ and $z H_i$ are connected by an edge in $P_J(\mathbb{Y})$. By Definition \ref{pc}, this is true if \begin{center}$d_Y(H_i, zH_i) \leq J$ for all $Y \neq H_i, zH_i $.\end{center} In view of (10), we will estimate $d^{\pi}_Y(H_i, zH_i)$.
\begin{figure}
\centering
\def\svgscale{0.4}
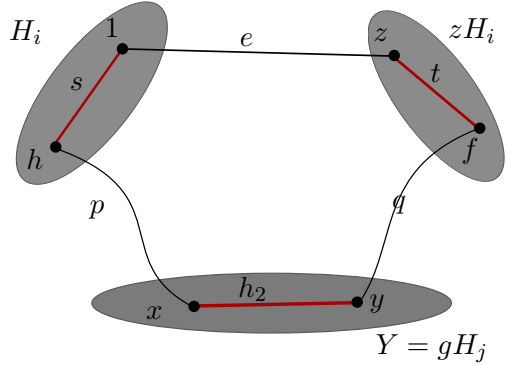
\caption{ Dealing with elements of $Z \cap K$ that represent elements of $\mathcal{H}$}
\label{incl}
\end{figure}

Let $\dzh(h, x) = \dzh(H_i, Y)$ and $\dzh(f, y) = \dzh(zH_i, Y)$ for some $h \in H_i, f \in zH_i$ and for some $x,y \in Y= gH_j$. Let $p$ be a geodesic connecting $h$ and $x$; and $q$ be a geodesic connecting $y$ and $f$. Let $h_2$ denote the edge connecting $x$ and $y$, labelled by an element of $H_j$. Similarly, let $s, t$ denote the edges connecting $h, 1$ and $z, f$ respectively, that are labelled by elements of $H_i$. Let $e$ denote the edge connecting $1$ and $z$, labelled by $z$. Consider the geodesic hexagon $W$ with sides $p, h_2, q, t, e, s$ (see Fig.\ref{incl}).

Arguing as in Lemma \ref{a3}), we can show that $h_2$ cannot be connected to $q$ , $p$, $s$ or $t$. Since  $z$ does not represent any element of $H_i$ for all $i$, $h_2$ cannot be connected to $e$. Thus, $h_2$ is isolated in $W$. By Lemma \ref{isolated}, $\widehat{d_j}(x, y) \leq 6C$. By Lemma \ref{bdd}, $$d_Y(H_i, zH_i) \leq d^{\pi}_Y(H_i, zH_i) \leq 14C + 2\xi.$$ So we conclude that by taking the constant $J$ to be sufficiently large so that Proposition \ref{K} holds and $J$ exceeds $14C +2\xi$, we can ensure that $z \in X$ and the arguments of the previous section still hold. \end{proof}

\begin{lem}\label{rep} There are only finitely many elements of $Z \cap K$ that can represent an element of $H_i$ for some $i \in \{1,2,...,n\}$. \end{lem}
\begin{proof} Let $z \in Z \cap K$ represent an element of $H_i$ for some $i=1,2,...,n$. Then in the Cayley graph $\Gamma(G, Z \sqcup \mathcal{H})$, we have a bigon between the elements $1$ and $h$, where one edge is labelled by $z$, and the other edge is labelled by an element of $H_i$, say $h_1$ (see Rem. \ref{bigons} and Fig.\ref{bigon}).
\begin{figure}
\centering
\def\svgscale{0.35}
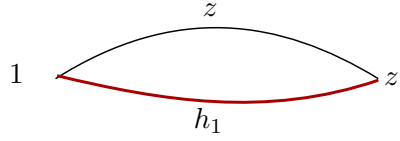
\caption{Bigons in the Cayley graph}
\label{bigon}
\end{figure}

This implies that $\widehat{d_i}(1,z) \leq 1$, so $\widetilde{d_i}(1,z) \leq 1$. But then $ z \in \widetilde{B_i}(1,1)$, i.e., the ball of radius 1 in the subgroup $H_i$  in the relative metric, centered at the identity. But this is a finite ball. Take $$\rho = \left| \bigcup_{i=1}^{n}  \widetilde{B_i}(1, 1)\right|.$$ Then $z$ has at most $\rho$ choices, which is finite. \end{proof}

By Lemma \ref{rep} and by selecting the constant $J$ as specified in Lemma \ref{notrep}, we conclude that the set $X$ from Section \ref{X} does not contain at most finitely many elements of $Z\cap K$. By adding these finitely many remaining elements of $Z \cap K$ to $X$, we obtain a new relative generating set $X'$ such that $|X' \Delta X| < \infty$. By Proposition \ref{symmdiff}, $\{H_1, H_2,...,H_n\} \hookrightarrow_h (K, X')$ and $ Z \cap K \subset X'$. By Remark \ref{change}, $\Gamma(K, X' \sqcup \mathcal{H})$ is also a quasi-tree. Thus $X'$ is the required set in the statement of Proposition \ref{prop}, which completes the proof. 

\subsection{Applications of Theorem \ref{3 condns}} \label{secappln}
In order to prove Theorem \ref{main}, we first need to recall the following definitions. 

\begin{defn}[Loxodromic element] Let $G$ be a group acting on a hyperbolic space $S$. An element $g\in G$ is called \textit{loxodromic} if the map $\mathbb{Z} \rightarrow S$ defined by $ n \rightarrow g^ns$ is a quasi-isometric embedding for some (equivalently, any) $s \in S$. \end{defn}

\begin{defn}[Elementary subgroup, Lemma 6.5 in \cite{DGO}] Let $G$ be a group acting acylindrically on a hyperbolic space $S$, $g \in G$ a loxodromic element. Then $g$ is contained in a unique maximal \textit{elementary subgroup $E(g)$} of $G$ given by
$$E(g) = \{h \in G \left| d_{Hau}(l, h(l)) < \infty \right.\},$$ where $l$ is a quasi-geodesic axis of $g$ in $S$. \end{defn}

\begin{cor}\label{finally} A  group $G$ is acylindrically hyperbolic if and only if G has an acylindrical and non-elementary action on a quasi-tree. \end{cor}

\begin{proof}
If $G$ has an acylindrical and non-elementary action on a quasi-tree, by Theorem \ref{ah}, $G$ is acylindrically hyperbolic. Conversely, let $G$ be acylindrically hyperbolic, with an acylindrical non-elementary action on a hyperbolic space $X$. Let $g$ be a loxodromic element for this action. By Lemma 6.5 of \cite{DGO} the elementary subgroup $E(g)$ is virtually cyclic and thus countable. By Theorem 6.8 of \cite{DGO}, $E(g)$ is hyperbolically embedded in $G$. Taking $K = G$ and $E(g)$ to be the hyperbolically embedded subgroup in the statement of Theorem \ref{3 condns} now gives us the result. Since $E(g)$ is non-degenerate, by \cite{ah}, Lemma 5.12, the resulting action of $G$ on the associated Cayley graph $\Gamma(G, X \sqcup E(g))$ is also non-elementary. \end{proof}

The following corollary is an immediate consequence of Theorem \ref{3 condns}.  

\begin{cor} \label{appln} Let $\{ H_1, H_2,..., H_n\}$ be a finite collection of countable subgroups of a group $G$ such that $\{H1, H_2,...,H_n\} \hookrightarrow_h G$. Let $K$ be a subgroup of $G$. If $H_i \leq K$ for all $i =1,2,...,n$, then $\{H_1, H_2,...,H_n\} \hookrightarrow_h K$. \end{cor}

\begin{defn}\label{cc} Let $(M, d)$ be a geodesic metric space, and $\epsilon >0$ a fixed constant. A subset $S \subset M$ is said to be \textit{$\epsilon$-coarsely connected} if there for any two points $x,y$ in $S$, there exist points $x_0 =x, x_1, x_2,...,x_{n-1}, x_n =y$ in $S$ such that  for all $i =0,...,n-1$, $$d(x_i, x_{i+1}) \leq \epsilon.$$ Further we say that $S$ is \textit{coarsely connected} if it is $\epsilon$-coarsely connected for some $\epsilon >0$. \end{defn} 

Recall that we denote the closed $\sigma$ neighborhood of $S$ by $S^{+\sigma}$.
\begin{defn} Let $(M, d)$ be a geodesic metric space, and $\sigma >0$ a fixed constant. A subset $S \subset M$ is said to be \textit{$\sigma$-quasi-convex} if for any two points $x,y$ in $S$, any geodesic connecting $x$ and $y$ is contained in  $S^{+\sigma}$. Further, we say that $S$ is \textit{quasi-convex} if it is $\sigma$-quasi-convex for some $\sigma >0$. \end{defn} 

\begin{cor}\label{newqc} Let $H$ be a finitely generated subgroup of an acylindrically hyperbolic group $G$. Then there exists a subset $X \subset G$ such that \begin{itemize} 
\item[(a)] $\Gamma(G, X)$ is hyperbolic, non-elementary and acylindrical 
\item[(b)] $H$ is quasi-convex in $\Gamma(G,X)$
\end{itemize} 
\end{cor} 

To prove the above corollary, we need the following two lemmas. 

\begin{lem}\label{intree} Let $T$ be a tree, and let $Q \subset T$ be $\epsilon$-coarsely connected. Then $Q$ is $\epsilon$-quasi-convex. \end{lem}
\begin{proof} Let $\epsilon >0$ be the constant from Definition \ref{cc}. Let $x,y$ be two points in $Q$, and $p$ be any geodesic between them. Then there exist points  $x_0 =x, x_1, x_2,...,x_{n-1}, x_n =y$ in $Q$ such that  for all $i =0,...,n-1$, $d(x_i, x_{i+1}) \leq \epsilon.$ Let $p_i$ denote the geodesic segments between $x_i$ and $x_{i+1}$for all $i =0,1,...,n-1$. Since $T$ is a tree, we must have that $$p \subseteq \bigcup_{i=0}^{n-1} p_i.$$ By definition, for all $i=0,1,...,n-1$, $p_i \subseteq B(x_i, \epsilon)$, the ball of radius $\epsilon$ centered at $x_i$. Since $x_i \in Q$ for all $i=0,1,...,n-1$, we obtain $$p_i \subseteq Q^{+\epsilon}.$$ This implies $p \subseteq Q^{+\epsilon}$. \end{proof}

\begin{lem}\label{inqtree} Let $\Gamma$ be a quasi-tree, and $S \subset \Gamma$ be coarsely connected. Then $S$ is quasi-convex.\end{lem}
\begin{proof} Let $T$ be a tree, and $d_\Gamma$ and $d_T$ denote distances in $\Gamma$ and $T$ respectively. Let $\delta >0$ be the hyperbolicity constant of $\Gamma$. Let $q \colon T \rightarrow \Gamma$ be a $(\lambda, C)$-quasi-isometry. i.e., $$ -C + \frac{1}{\lambda} d_T(a,b) \leq d_\Gamma(q(a), q(b)) \leq \lambda d_T(a, b) + C.$$ Let $\epsilon >0$ be the constant from Definition \ref{cc} for S. Set $Q = q^{-1}(S)$. Then $Q \subset T$. It is easy to check that $Q$ is $\rho$-coarsely connected with constant $\rho = \lambda (\epsilon +C)$.  By Lemma \ref{intree}, $Q$ is $\rho$-quasi-convex. \\
Let $x,y$ be two points in $S$, and $p$ be a geodesic between them. Choose points $a, b$ in $Q$ such that $q(a) = x$ and $q(b) =y$. Let $r$ denote the (unique) geodesic in $T$ between $a$ and $b$. Since $Q$ is $\rho$-quasi-convex, we have $$r \subseteq Q^{+\rho}.$$ Set $\sigma = \lambda \rho +C$. Then $$q(r) \subseteq S^{+\sigma}.$$ Further $q \circ r$ is a quasi-geodesic between $x$ and $y$. By Lemma \ref{bottlenecklemma}, there exists a constant $R (=R(\lambda, C, \delta))$ such that $q(r)$ and $p$ are Hausdorff distance less than $R$ from each other. This implies that $p \subseteq S^{+(R + \sigma)}$. Thus $S$ is quasi-convex. \end{proof} 

\begin{proof}[Proof of Corollary \ref{newqc}]
By Corollary \ref{finally}, there exists a generating set $X$ of $G$ such that $\Gamma(G, X)$ is a quasi-tree (hence hyperbolic), and the action of $G$ on $\Gamma(G,X)$ is acylindrical and non-elementary. Let $d_X$ denote the metric on $\Gamma(G,X)$ induced by the generating set $X$. Let $H = \langle x_1, x_2,...,x_n \rangle$. Set $$\epsilon = max \{ d_X(1, x_i^{\pm 1}) \mid i =1,2,...,n\}.$$ We claim that $H$ is coarsely connected with constant $\epsilon$. Indeed if $u, v$ are elements of $H$, then $u^{-1}v = \prod_{j =1}^{k} w_j$, where $w_j \in \{ x_1^{\pm 1},..., x_n^{\pm 1} \}$. Set $$z_0 = u, z_1 = uw_1,...,z_{k-1} = uw_1w_2...w_{k-1}, z_k = v.$$ Clearly $z_i \in H$ for all $i=0,2,...,k-1$. Further $$d_X(z_i, z_{i+1}) = d_X(1, w_{i+1}) \leq \epsilon$$ for all $i=0,1,2,...,k-1$. By Lemma \ref{inqtree}, $H$ is quasi-convex in $\Gamma(G, X)$. \end{proof}

\vspace{10pt}
\textbf{Sahana Balasubramanya} \\
Department of Mathematics, Vanderbilt University\\
Nashville, TN -37240, U.S.A\\
Email : sahana.balasubramanya@vanderbilt.edu 
\end{document}